\documentclass[oneside,english,a4paper]{amsart}
\usepackage[T1]{fontenc}
\usepackage[latin9]{inputenc}
\usepackage{amstext}
\usepackage{amsthm}
\usepackage{amssymb}
\makeatletter
\numberwithin{equation}{section}
\numberwithin{figure}{section}
\theoremstyle{plain}
\newtheorem{thm}{\protect\theoremname}
  \theoremstyle{remark}
    \newtheorem*{nota}{Notation}
  \newtheorem{rem}[thm]{\protect\remarkname}
  \theoremstyle{plain}
  \newtheorem{lem}[thm]{\protect\lemmaname}
  \theoremstyle{definition}
  \newtheorem{defn}[thm]{\protect\definitionname}
  \theoremstyle{plain}
  \newtheorem{prop}[thm]{\protect\propositionname}

\makeatother

\usepackage{babel}
  \providecommand{\definitionname}{Definition}
  \providecommand{\lemmaname}{Lemma}
  \providecommand{\propositionname}{Proposition}
  \providecommand{\remarkname}{Remark}
\providecommand{\theoremname}{Theorem}

\begin{document}

\title[Compactness for scalar-flat metrics on umbilic
boundary manifolds]{A compactness result for scalar-flat metrics on manifolds with umbilic
boundary}

\author{Marco Ghimenti}
\address{M. Ghimenti, \newline Dipartimento di Matematica Universit\`a di Pisa
Largo B. Pontecorvo 5, 56126 Pisa, Italy}
\email{marco.ghimenti@unipi.it}

\author{Anna Maria Micheletti}
\address{A. M. Micheletti, \newline Dipartimento di Matematica Universit\`a di Pisa
Largo B. Pontecorvo 5, 56126 Pisa, Italy}
\email{a.micheletti@dma.unipi.it.}

\begin{abstract}
Let $(M,g)$ a compact Riemannian $n$-dimensional manifold with umbilic
boundary. It is well know that, under certain hypothesis, in the conformal
class of $g$ there are scalar-flat metrics that have $\partial M$
as a constant mean curvature hypersurface. In this paper we prove
that these metrics are a compact set, provided $n=8$ and the Weyl
tensor of the boundary is always different from zero, or if $n>8$
and the Weyl tensor of $M$ is always different from zero on the boundary.
\end{abstract}


\maketitle

\section{Introduction}

Let $(M,g)$ be a $n$-dimensional ($n\ge3$) compact Riemannian manifold
with boundary $\partial M$. In \cite{Es,Es2} J. Escobar investigated
the question of finding a conformal metric $\tilde{g}=u^{\frac{4}{n-2}}g$
for which $M$ has constant scalar curvature and $\partial M$ as
constant mean curvature hypersurface. From a PDEs point of view, this
is equivalent to the existence of a positive solution to the equation
\begin{equation}
\left\{ \begin{array}{cc}
L_{g}u=ku^{\frac{n+2}{n-2}} & \text{ in }M\\
B_{g}u=cu^{\frac{n}{n-2}} & \text{ on }\partial M
\end{array}\right.\label{eq:prob*}
\end{equation}
where $L_{g}u=\Delta_{g}u-\frac{n-2}{4(n-1)}R_{g}u$ and $B_{g}u=-\frac{\partial}{\partial\nu}u-\frac{n-2}{2}h_{g}u$
are respectively the conformal Laplacian and the conformal boundary
operator, $R_{g}$ is the scalar curvature of the manifold, $h_{g}$
is the mean curvature of the $\partial M$ and $\nu$ is the outer
normal with respect to $\partial M$. The motivation to study this
question arises from the classical Yamabe problem which consists of
finding a constant scalar curvature metric, conformal to a given metric
$g$ on a compact Riemannian manifold without boundary. By the works
of Yamabe, Trudinger, Aubin, Schoen \cite{Au1,S,T,Y} the original
problem was settled.

If a solution $u$ of Problem (\ref{eq:prob*}) exists, then the metric
$\tilde{g}=u^{\frac{4}{n-2}}g$ has constant scalar curvature $\frac{k(n-2)}{4(n-1)}$
and the boundary has mean curvature $c$. Problem (\ref{eq:prob*})
has been studied by many authors, see the recent paper of Disconzi,
Khuri \cite{DK} and the survey of Marques \cite{M2} for a list of references. For the case $c=0$ we
limit ourselves to cite among others \cite{ALM} and references therein. 

In this paper we consider the case of zero scalar curvature which
is particularly interesting because it is a higher-dimensional generalization
of the well known Riemann mapping Theorem and it leads to a linear
equation on the interior of $M$ with a critical nonlinear boundary
condition of Neumann type.

Namely, we are interested to positive solution of the equation 
\begin{equation}
\left\{ \begin{array}{cc}
L_{g}u=0 & \text{ in }M\\
B_{g}u+(n-2)u^{\frac{n}{n-2}}=0 & \text{ on }\partial M
\end{array}\right..\label{eq:Pconf}
\end{equation}
Solutions of (\ref{eq:P-conf}) are critical points of the functional
quotient
\[
Q(u):=\frac{\int\limits _{M}\left(|\nabla u|^{2}+\frac{n-2}{4(n-1)}R_{g}u^{2}\right)dv_{g}+\int\limits _{\partial M}\frac{n-2}{2}h_{g}u^{2}d\sigma_{g}}{\left(\int\limits _{\partial M}|u|^{\frac{2(n-1)}{n-2}}d\sigma_{g}\right)^{\frac{n-2}{n-1}}}.
\]
 In \cite{Es} Escobar introduced, in analogy with the classical Yamabe
problem, the quotient
\[
Q(M,\partial M):=\inf\left\{ Q(u)\ :\ u\in H^{1}(M),\ u\not\equiv0\text{ on }\partial M\right\} 
\]
which always satisfies the fundamental estimate
\begin{equation}
Q(M,\partial M)\le Q(\mathbb{B}^{n},\mathbb{S}^{n-1})\label{eq:ineq}
\end{equation}
where $\mathbb{B}^{n}$ is the unit ball in $\mathbb{R}^{n}$ endowed
with euclidean metric. Inequality (\ref{eq:ineq}) is important since
if it strict inequality holds, then a solution of (\ref{eq:Pconf})
exists.

When $(M,g)$ is not conformally equivalent to $(\mathbb{B}^{n},g_{\mathbb{R}^{n}})$,
existence results are proved by Escobar \cite{Es}, Marques \cite{M1},
Almaraz \cite{A3}, Chen \cite{ch}, Mayer and Ndiaye \cite{MN}.

Once the existence of solutions of (\ref{eq:Pconf}) is settled, it
is natural to study the compactness of the full set of solutions.
If $Q(M,\partial M)\le0$ the solution is unique up to a constant
factor. The situation turns out to be delicate if $Q(M,\partial M)>0$
and the underlying manifold is not the euclidean ball (in the case
of the euclidean ball the set of solution is known to be non compact).
Compactness has be proven by Felli and Ould Ahmedou in \cite{FA}
for any dimension $n\ge3$ in the case of locally conformally flat
manifolds with umbilic boundary and by Almaraz in \cite{Al} when
$n\ge7$ and the trace-free second fundamental form in non zero everywhere
on $\partial M$. An example of non compactness is given for $n\ge25$
and manifolds with umbilic boundary in \cite{A2}. We recall that
the boundary of $M$ is called \emph{umbilic} if the trace-free second
fundamental form of $\partial M$ is zero everywhere.

In the present work we are interested in the compactness of the set
of positive solutions to
\begin{equation}
\left\{ \begin{array}{cc}
L_{g}u=0 & \text{ in }M\\
B_{g}u+(n-2)u^{p}=0 & \text{ on }\partial M
\end{array}\right.\label{eq:Prob-2}
\end{equation}
where $1\le p\le\frac{n}{n-2}$ and the boundary of $M$ is umbilic.
Our main result is the following.
\begin{thm}
\label{thm:main}Let $(M,g)$ a smooth, $n$-dimensional Riemannian
manifold of positive type with regular umbilic boundary $\partial M$.
Suppose that $n>8$ and that the Weyl tensor $W_{g}$ is not vanishing
on $\partial M$ or suppose that $n=8$ and that the Weyl tensor referred
to the boundary $\bar{W}_{g}$ is not vanishing on $\partial M$.
Then, given $\bar{p}>1$, there exists a positive constant $C$ such
that, for any $p\in\left[\bar{p},\frac{n}{n-2}\right]$ and for any
$u>0$ solution of (\ref{eq:Prob-2}), it holds 
\[
C^{-1}\le u\le C\text{ and }\|u\|_{C^{2,\alpha}(M)}\le C
\]
for some $0<\alpha<1$. The constant $C$ does not depend on $u,p$. 
\end{thm}
The proof is based on a local argument with Pohozaev type identity. This strategy was first introduced by Schoen \cite{S} for a manifold 
without boundary. 
In this paper we avoid the use of any positive mass assumption: a crucial step
is to provide a sharp correction term (see Lemma \ref{lem:vq}, Lemma
\ref{lem:coreLemma} and Proposition \ref{eq:stimaW1}) for the usual
approximation of a rescaled solution by a bubble around an isolated
simple blow up point (see Definition \ref{def:isolated}). The idea
of using a suitable correction term of a bubble to obtain refined
point-wise blow up estimates was used in \cite{B,HV,KMS}, in the case
of manifold without boundary, and in \cite{Al}, in the case of manifold
with boundary.

The compactness issue is closely related to the existence of blowing
up solution for small perturbation of (\ref{eq:Pconf}). In this direction
there are some result of noncompactness for the perturbed problem
if the linear perturbation of the mean curvature on the boundary is
strictly positive everywhere (see \cite{GMP,GMP2}). Then we do not
have the stability of compactness result under a small positive linear
perturbation of the boundary condition.

A key observation is that our correction term allows us to obtain
the vanishing of the Weyl tensor on the boundary (see Proposition
\ref{prop:7.1}).

The paper is organized as follows. After some preliminaries, in Section
\ref{sec:blowuppoints} we recall the notions of isolated and isolated
simple blow up point, and some well known basic properties related
to these points. In Section \ref{sec:estimates}, and in particular
in Proposition \ref{eq:stimaW1}, we give a crucial estimate for a
blowing up sequence of solutions near an isolated simple blow up point,
using the sharp correction term defined in Lemma \ref{lem:vq}. Then,
in Section \ref{sec:Poho} and in Section \ref{sec:Sign}, after presenting
a Pohozaev type identity, we provide a sign estimate of the terms
of Pohozaev identity near an isolated simple blow up point, and by
this result we prove the vanishing of Weyl tensor at any isolated
simple blow up point (Proposition \ref{prop:7.1}). In Section \ref{sec:splitting}
we reduce our analysis to the case of an isolated simple blow up points.
and finally in Section \ref{sec:Proof} we prove our compactness result.

\section{Preliminaries and notations}
\begin{nota}
We collect here our main notations. We will use the indices $1\le i,j,k,m,p,r,s\le n-1$
and $1\le a,b,c,d\le n$. Moreover we use the Einstein convention
on repeated indices. We denote by $g$ the Riemannian metric, by $R_{abcd}$
the full Riemannian curvature tensor, by $R_{ab}$ the Ricci tensor
and by $R_{g}$ the scalar curvature of $(M,g)$; moreover the Weyl
tensor of $(M,g)$ will be denoted by $W_{g}$. The bar over an object
(e.g. $\bar{W}_{g}$) will means the restriction to this object to
the metric of $\partial M$. By $-\Delta_{g}$ we denote the Laplace-Beltrami
operator on $(M,g)$ and we will often use the common notation for
conformal Laplacian $L_{g}=-\Delta_{g}+\frac{n-2}{4(n-1)}R_{g}$ and
the conformal boundary operator $B_{g}=\frac{\partial}{\partial\nu}+\frac{n-2}{2}h_{g}$,
where $\nu$ is the outward normal to $\partial M$. Finally, on the
half space $\mathbb{R}_{+}^{n}=\left\{ y=(y_{1},\dots,y_{n-1},y_{n})\in\mathbb{R}^{n},\!\ y_{n}\ge0\right\} $
we set $B_{r}(y_{0})=\left\{ y\in\mathbb{R}^{n},\!\ |y-y_{0}|\le r\right\} $
and $B_{r}^{+}(y_{0})=B_{r}(y_{0})\cap\left\{ y_{n}>0\right\} $.
When $y_{0}=0$ we will use simply $B_{r}=B_{r}(y_{0})$ and $B_{r}^{+}=B_{r}^{+}(y_{0})$.
On the half ball $B_{r}^{+}$ we set $\partial'B_{r}^{+}=B_{r}^{+}\cap\partial\mathbb{R}_{+}^{n}=B_{r}^{+}\cap\left\{ y_{n}=0\right\} $
and $\partial^{+}B_{r}^{+}=\partial B_{r}^{+}\cap\left\{ y_{n}>0\right\} $.
On $\mathbb{R}_{+}^{n}$ we will use the following decomposition of
coordinates: $(y_{1},\dots,y_{n-1},y_{n})=(\bar{y},y_{n})=(z,t)$
where $\bar{y},z\in\mathbb{R}^{n-1}$ and $y_{n},t\ge0$.

Finally, fixed a point $q\in\partial M$, we denote by $\psi_{q}:B_{r}^{+}\rightarrow M$
the Fermi coordinates centered at $q$. We denote by $B_{g}^{+}(q,r)$
the image of $\psi_{q}(B_{r}^{+})$. When no ambiguity is possible,
we will denote $B_{g}^{+}(q,r)$ simply by $B_{r}^{+}$, omitting
the chart $\psi_{q}$.
\end{nota}

We can work with a slightly more general problem
\begin{equation}
\left\{ \begin{array}{cc}
L_{g}u=0 & \text{ in }M\\
B_{g}u+(n-2)f^{-\tau}u^{p}=0 & \text{ on }\partial M
\end{array}\right.\label{eq:Prob-3}
\end{equation}
where $\tau=\frac{n}{n-2}-p$, $p\in\left[\bar{p},\frac{n}{n-2}\right]$
for some fixed $\bar{p}>1$, and $f>0$. The reason to work with this
equation instead of equation (\ref{eq:Prob-2}) is that equation (\ref{eq:Prob-3})
has an important conformal invariance property.

Since the boundary $\partial M$of $M$ is umbilic, it is well know
the existence of a conformal metric related to $g$ and the existence
of the conformal Fermi coordinates, which will simplify the future
computations. 

Given $q\in\partial M$ there exists a conformally related metric
$\tilde{g}_{q}=\Lambda_{q}^{\frac 4{n-2}}g$ such that some geometric quantities
at $q$ have a simpler form which will be summarized in the next claim.
Moreover 
\[
\Lambda_{q}(q)=1,\ \frac{\partial\Lambda_{q}}{\partial y_{k}}(q)=0\text{ for all }k=1,\dots,n-1.
\]
 Set $\tilde{u}_{q}=\Lambda_{q}^{-1}u$ and $\tilde{f}_{q}=\Lambda_{q}f$
it holds 
\begin{equation}
\left\{ \begin{array}{cc}
L_{\tilde{g}_{q}}\tilde{u}_{q}=0 & \text{ in }M\\
B_{\tilde{g}_{q}}\tilde{u}_{q}+(n-2)\tilde{f}_{q}^{-\tau}\tilde{u}_{q}^{p}=0 & \text{ on }\partial M
\end{array}\right..\label{eq:P-conf}
\end{equation}
In the following we study equation (\ref{eq:P-conf}) and in order
to simplify notations, we will omit the \emph{tilda} symbol and we
will omit $\psi_{x_{i}}$ whenever is not needed, so we will write
\[
y\in B_{r}^{+}\text{ instead of }\psi_{q}(y)\in M;\ 0\text{ instead of }q=\psi_{q}(0);\ u\text{ instead of }u\circ\psi_{q}
\]
where $\psi_{q}:B_{r}^{+}\rightarrow M$ are the Fermi conformal coordinates
centered at $q$. 
\begin{rem}
\label{rem:confnorm}In Fermi conformal coordinates around $q\in\partial M$,
it holds (see \cite{M1})
\begin{equation}
|\text{det}g_{q}(y)|=1+O(|y|^{n})\label{eq:|g|}
\end{equation}
\begin{eqnarray}
|h_{ij}(y)|=O(|y^{4}|) &  & |h_{g}(y)|=O(|y^{4}|)\label{eq:hij}
\end{eqnarray}
\begin{align}
g_{q}^{ij}(y)= & \delta^{ij}+\frac{1}{3}\bar{R}_{ikjl}y_{k}y_{l}+R_{ninj}y_{n}^{2}\label{eq:gij}\\
 & +\frac{1}{6}\bar{R}_{ikjl,m}y_{k}y_{l}y_{m}+R_{ninj,k}y_{n}^{2}y_{k}+\frac{1}{3}R_{ninj,n}y_{n}^{3}\nonumber \\
 & +\left(\frac{1}{20}\bar{R}_{ikjl,mp}+\frac{1}{15}\bar{R}_{iksl}\bar{R}_{jmsp}\right)y_{k}y_{l}y_{m}y_{p}\nonumber \\
 & +\left(\frac{1}{2}R_{ninj,kl}+\frac{1}{3}\text{Sym}_{ij}(\bar{R}_{iksl}R_{nsnj})\right)y_{n}^{2}y_{k}y_{l}\nonumber \\
 & +\frac{1}{3}R_{ninj,nk}y_{n}^{3}y_{k}+\frac{1}{12}\left(R_{ninj,nn}+8R_{nins}R_{nsnj}\right)y_{n}^{4}+O(|y|^{5})\nonumber 
\end{align}
\begin{equation}
\bar{R}_{g_{q}}(y)=O(|y|^{2})\text{ and }\partial_{ii}^{2}\bar{R}_{g_{q}}=-\frac{1}{6}|\bar{W}|^{2}\label{eq:Rii}
\end{equation}
\begin{equation}
\partial_{tt}^{2}\bar{R}_{g_{q}}=-2R_{ninj}^{2}-2R_{ninj,ij}\label{eq:Rtt}
\end{equation}
\begin{equation}
\bar{R}_{kl}=R_{nn}=R_{nk}=R_{nn,kk}=0\label{eq:Ricci}
\end{equation}
\begin{equation}
R_{nn,nn}=-2R_{nins}^{2}.\label{eq:Rnnnn}
\end{equation}
All the quantities above are calculate in $q\in\partial M$, unless
otherwise specified.
\end{rem}
We set ${\displaystyle U(y):=\frac{1}{\left[(1+y_{n})^{2}+|\bar{y}|^{2}\right]^{\frac{n-2}{2}}}}$
to be the standard bubble. The function $U$ solves the problem 
\begin{equation}
\left\{ \begin{array}{cc}
\Delta U=0 & \text{ in }\mathbb{R}_{+}^{n}\\
\frac{\partial U}{\partial y_{n}}+(n-2)U^{\frac{n}{n-2}}=0 & \text{ on }\partial\mathbb{R}_{+}^{n}
\end{array}\right..\label{eq:ProbBubble}
\end{equation}
If we linearize Problem (\ref{eq:ProbBubble}) around the function
$U$, we have that all the solutions of the linearized problem are
generated by the functions
\begin{equation}
j_{l}:=\partial_{l}U=-(n-2)\frac{y_{l}}{\left[(1+y_{n})^{2}+|\bar{y}|^{2}\right]^{\frac{n}{2}}}\label{eq:jl}
\end{equation}
\begin{equation}
j_{n}:=y^{b}\partial_{b}U+\frac{n-2}{2}U=-\frac{n-2}{2}\frac{|y|^{2}-1}{\left[(1+y_{n})^{2}+|\bar{y}|^{2}\right]^{\frac{n}{2}}}.\label{eq:jn}
\end{equation}
Finally, we have
\[
\partial_{k}\partial_{l}U=(n-2)\left\{ \frac{ny_{l}y_{k}}{\left[(1+y_{n})^{2}+|\bar{y}|^{2}\right]^{\frac{n+2}{2}}}-\frac{\delta^{kl}}{\left[(1+y_{n})^{2}+|\bar{y}|^{2}\right]^{\frac{n}{2}}}\right\} .
\]
In the following Lemma we introduce the function $\gamma_{q}$ as
the solution of a certain linear problem. This function $\gamma_{q}$
plays a fundamental role in this paper: by this choice of $\gamma_{q}$
we are able to cancel the term of second order in formula (\ref{eq:Qi-parziale}),
which is crucial to obtain Lemma \ref{lem:coreLemma}. Also, the estimates
of Proposition \ref{prop:stimawi} and of Lemma \ref{lem:R(U,vq)}
depend on the properties of function $\gamma_{q}$. The proof of the
following Lemma is analogous to \cite[Lemma 3]{GMP} and \cite[Proposition 5.1]{Al}.
However, we rewrite the proof in the appendix.
\begin{lem}
\label{lem:vq}Assume $n\ge5$. Given a point $q\in\partial M$, there
exists a unique $\gamma_{q}:\mathbb{R}_{+}^{n}\rightarrow\mathbb{R}$
a solution of the linear problem 
\begin{equation}
\left\{ \begin{array}{ccc}
-\Delta\gamma=\left[\frac{1}{3}\bar{R}_{ijkl}(q)y_{k}y_{l}+R_{ninj}(q)y_{n}^{2}\right]\partial_{ij}^{2}U &  & \text{on }\mathbb{R}_{+}^{n}\\
\frac{\partial\gamma}{\partial y_{n}}=-nU^{\frac{2}{n-2}}\gamma &  & \text{on }\partial\mathbb{R}_{+}^{n}
\end{array}\right.\label{eq:vqdef}
\end{equation}
which is $L^{2}(\mathbb{R}_{+}^{n})$-orthogonal to the functions
$j_{1},\dots,j_{n}$ defined in (\ref{eq:jl}) and (\ref{eq:jn}).

Moreover it holds
\begin{equation}
|\nabla^{\tau}\gamma_{q}(y)|\le C(1+|y|)^{4-\tau-n}\text{ for }\tau=0,1,2;\label{eq:gradvq}
\end{equation}
\begin{equation}
\int_{\mathbb{R}_{+}^{n}}\gamma_{q}\Delta\gamma_{q}dy\le0;\label{new}
\end{equation}

\begin{equation}
\int_{\partial\mathbb{R}_{+}^{n}}U^{\frac{n}{n-2}}(t,z)\gamma_{q}(t,z)dz=0;\label{eq:Uvq}
\end{equation}
\begin{equation}
\gamma_{q}(0)=\frac{\partial\gamma_{q}}{\partial y_{1}}(0)=\dots=\frac{\partial\gamma_{q}}{\partial y_{n-1}}(0)=0.\label{eq:dervq}
\end{equation}

Finally the map $q\mapsto\gamma_{q}$ is $C^{2}(\partial M)$.
\end{lem}

\section{Isolated and isolated simple blow up points\label{sec:blowuppoints}}

In this section we will define two particular kind of blow up points,
and we collect a series of results that focus on the asymptotic behavior
of these blow up points. These results are now quite standard, so
we will only collect the claims, while for the proofs we refer to
\cite{Al,FA,HL,M3}. 

Let $\left\{ u_{i}\right\} _{i}$ be a sequence of positive solution
to 
\begin{equation}
\left\{ \begin{array}{cc}
L_{g_{i}}u=0 & \text{ in }M\\
B_{g_{i}}u+(n-2)f_{i}^{-\tau_{i}}u^{p_{i}}=0 & \text{ on }\partial M
\end{array}\right..\label{eq:Prob-i}
\end{equation}
where $p_{i}\in\left[\bar{p},\frac{n}{n-2}\right]$ for some fixed
$\bar{p}>1$, $\tau_{i}=\frac{n}{n-2}-p_{i}$, $f_{i}\rightarrow f$
in $C_{\text{loc}}^{1}$ for some positive function $f$ and $g_{i}\rightarrow g_{0}$
in the $C_{\text{loc}}^{3}$ topology.
\begin{defn}
\label{def:blowup}We say that $x_{0}\in\partial M$ is a blow up
point for the sequence $u_{i}$ of solutions of (\ref{eq:Prob-i})
if there is a sequence $x_{i}\in\partial M$ such that 
\begin{enumerate}
\item $x_{i}\rightarrow x_{0}$;
\item $x_{i}$ is a local maximum point of $\left.u_{i}\right|_{\partial M}$
;
\item $u_{i}(x_{i})\rightarrow+\infty.$
\end{enumerate}
\end{defn}
Shortly we say that $x_{i}\rightarrow x_{0}$ is a blow up point for
$\left\{ u_{i}\right\} _{i}$. 

Given $x_{i}\rightarrow x_{0}$ a blow up point for $\left\{ u_{i}\right\} _{i}$,
we set
\[
M_{i}:=u_{i}(x_{i})
\]

\begin{defn}
\label{def:isolated}We say that $x_{i}\rightarrow x_{0}$ is an \emph{isolated}
blow up point for $\left\{ u_{i}\right\} _{i}$ if $x_{i}\rightarrow x_{0}$
is a blow up point for $\left\{ u_{i}\right\} _{i}$ and there exist
two constants $\rho,C>0$ such that
\[
u_{i}(x)\le Cd_{\bar{g}}(x,x_{i})^{-\frac{1}{p_{i-1}}}\text{ for all }x\in\partial M\smallsetminus\left\{ x_{i}\right\} ,\ d_{\bar{g}}(x,x_{i})<\rho.
\]
Here $\bar{g}$ denotes the metric on the boundary induced by $g$
and $d_{\bar{g}}(\cdot,\cdot)$ is the geodesic distance on the boundary
between two points.

We recall the following result
\end{defn}
\begin{prop}
\label{prop:4.1}Let $x_{i}\rightarrow x_{0}$ is an \emph{isolated}
blow up point for $\left\{ u_{i}\right\} _{i}$ and $\rho$ as in
Definition \ref{def:isolated}. We set 
\[
v_{i}(y)=M_{i}^{-1}(u_{i}\circ\psi_{i})(M_{i}^{1-p_{i}}y),\text{ for }y\in B_{\rho M_{i}^{p_{i}-1}}^{+}(0).
\]
Then, given $R_{i}\rightarrow\infty$ and $\beta_{i}\rightarrow0$,
up to subsequences, we have
\begin{enumerate}
\item $|v_{i}-U|_{C^{2}\left(B_{R_{i}}^{+}(0)\right)}<\beta_{i}$;
\item ${\displaystyle \lim_{i\rightarrow\infty}\frac{R_{i}}{\log M_{i}}=0}$;
\item ${\displaystyle \lim_{i\rightarrow\infty}p_{i}=\frac{n}{n-2}}$.
\end{enumerate}
\end{prop}
Given $x_{i}\rightarrow x_{0}$ an isolated blow up point for $\left\{ u_{i}\right\} _{i}$,
and given $\psi_{i}:B_{\rho}^{+}(0)\rightarrow M$ the Fermi coordinates
centered at $x_{i}$, we define the spherical average of $u_{i}$
as
\[
\bar{u}_{i}(r)=\frac{2}{\omega_{n-1}r^{n-1}}\int_{\partial^{+}B_{r}^{+}}u_{i}\circ\psi_{i}d\sigma_{r}
\]
and
\[
w_{i}(r):=r^{-\frac{1}{p_{i}-1}}\bar{u}_{i}(r)
\]
for $0<r<\rho.$
\begin{defn}
\label{def:isolatedsimple}We say that $x_{i}\rightarrow x_{0}$ is
an \emph{isolated simple} blow up point for $\left\{ u_{i}\right\} _{i}$
solutions of (\ref{eq:Prob-i}) if $x_{i}\rightarrow x_{0}$ is an
isolated blow up point for $\left\{ u_{i}\right\} _{i}$ and there
exists $\rho$ such that $w_{i}$ has exactly one critical point in
the interval $(0,\rho)$. 
\end{defn}
One can prove that is $x_{i}\rightarrow x_{0}$ is an isolated simple
blow up point for $\left\{ u_{i}\right\} _{i}$, and if $R_{i}\rightarrow+\infty,$
then 
\[
w_{i}'(r)<0\text{ for all }r\in[R_{i}M_{i}^{1-p_{i}},\rho).
\]
This allows to compare this definition of isolated simple blow up
point with the other one present in literature (see, e.g. \cite{FA}).
In fact, in light of Proposition \ref{prop:4.1}, if $x_{i}\rightarrow x_{0}$
is an isolated blow up point for $\left\{ u_{i}\right\} _{i}$ then
the function $r\rightarrow r^{\frac{1}{p_{i}-1}}\bar{u}_{i}(r)$ has
exactly one critical point in $(0,R_{i}M_{i}^{1-p_{i}})$ and the
derivative is negative right after the critical point.
\begin{prop}
\label{prop:Lemma 4.4}Let $x_{i}\rightarrow x_{0}$ be an isolated
simple blow up point for $\left\{ u_{i}\right\} _{i}$ and let $\eta$
small. Then there exist $C,\rho>0$ such that 
\[
M_{i}^{\lambda_{i}}|\nabla^{k}u_{i}(\psi_{i}(y))|\le C|y|^{2-k-n+\eta}
\]
for $y\in B_{\rho}^{+}(0)\smallsetminus\left\{ 0\right\} $ and $k=0,1,2$.
Here $\lambda_{i}=(p_{i}-1)(n-2-\eta)-1$.
\end{prop}
\begin{prop}
\label{prop:4.3}Let $x_{i}\rightarrow x_{0}$ be an isolated simple
blow up point for $\left\{ u_{i}\right\} _{i}$. Then there exist
$C,\rho>0$ such that 
\begin{enumerate}
\item $M_{i}u_{i}(\psi_{i}(y))\le C|y|^{2-n}$ for all $y\in B_{\rho}^{+}(0)\smallsetminus\left\{ 0\right\} $;
\item $M_{i}u_{i}(\psi_{i}(y))\ge C^{-1}G_{i}(y)$ for all $y\in B_{\rho}^{+}(0)\smallsetminus B_{r_{i}}^{+}(0)$
where $r_{i}:=R_{i}M_{i}^{1-p_{i}}$ and $G_{i}$ is the Green\textquoteright s
function which solves
\[
\left\{ \begin{array}{ccc}
L_{g_{i}}G_{i}=0 &  & \text{in }B_{\rho}^{+}(0)\smallsetminus\left\{ 0\right\} \\
G_{i}=0 &  & \text{on }\partial^{+}B_{\rho}^{+}(0)\\
B_{g_{i}}G_{i}=0 &  & \text{on }\partial'B_{\rho}^{+}(0)\smallsetminus\left\{ 0\right\} 
\end{array}\right.
\]
\end{enumerate}
and $|y|^{n-2}G_{i}(y)\rightarrow1$ as $|z|\rightarrow0$.
\end{prop}
Let us notice that, by Proposition \ref{prop:4.1} and by Proposition
\ref{prop:4.3} we have that, if $x_{i}\rightarrow x_{0}$ is an isolated
simple blow up point for $\left\{ u_{i}\right\} _{i}$, then, given
$v_{i}$ as in Proposition \ref{prop:4.1} it holds
\[
v_{i}\le CU\text{ in }B_{\rho M_{i}^{p_{i}-1}}^{+}(0).
\]

\section{\label{sec:estimates}Blow up estimates}

In this section $x_{i}\rightarrow x_{0}$ is an isolated simple blow
up point for a sequence $\left\{ u_{i}\right\} _{i}$ of solutions
of (\ref{eq:Prob-i}). We will work in the conformal normal coordinates
in a neighborhood of $x_{i}$.

Set $\tilde{u}_{i}=\Lambda_{x_{i}}^{-1}u_{i}$ we define
\begin{equation}
\delta_{i}=\tilde{u}_{i}^{1-p_{i}}(x_{i})=u_{i}^{1-p_{i}}(x_{i})=M_{i}^{1-p_{i}},\label{eq:deltai}
\end{equation}
since $\Lambda_{x_{i}}(x_{i})=1$. 

We have that $x_{i}\rightarrow x_{0}$ is also an isolated blow up
point for the function $\tilde{u}_{i}$ and the estimates of Proposition
\ref{prop:4.3} hold since we have uniform control on the conformal
factor $\Lambda_{i}$. In the following we simply omit the \emph{tilde}
symbol unless otherwise specified. 

Set 
\[
v_{i}(y):=\delta_{i}^{\frac{1}{p_{i}-1}}u_{i}(\delta_{i}y)\text{ for }y\in B_{\frac{R}{\delta_{i}}}^{+}(0),
\]
we know that $v_{i}$ satisfies 
\begin{equation}
\left\{ \begin{array}{cc}
L_{\hat{g}_{i}}v_{i}=0 & \text{ in }B_{\frac{R}{\delta_{i}}}^{+}(0)\\
B_{\hat{g}_{i}}v_{i}+(n-2)\hat{f}_{i}^{-\tau_{i}}v_{i}^{p_{i}}=0 & \text{ on }\partial'B_{\frac{R}{\delta_{i}}}^{+}(0)
\end{array}\right.\label{eq:Prob-hat}
\end{equation}
where $\hat{g}_{i}:=\tilde{g}_{i}(\delta_{i}y)=\Lambda_{x_{i}}^{\frac{4}{n-2}}(\delta_{i}y)g(\delta_{i}y)$,
$\hat{f}_{i}(y)=f_{i}(\delta_{i}y)$, $f_{i}=\Lambda_{x_{i}}f\rightarrow\Lambda_{x_{0}}f$
and $\tau_{i}=\frac{n}{n-2}-p_{i}$. 

Our aim is to provide by Lemma \ref{lem:vq} a sharp correction term
for the usual approximation of the rescaled solution $v$ by $U$,
near an isolated simple blow up point $x_{i}\rightarrow x_{0}$. This
result is obtained in Proposition \ref{prop:stimawi} at the end of
this section. First, we need two lemmas.
\begin{lem}
\label{lem:coreLemma}Assume $n\ge8$. Let $\gamma_{x_{i}}$ be defined
in (\ref{eq:vqdef}). There exist $R,C>0$ such that 
\[
|v_{i}(y)-U(y)-\delta_{i}^{2}\gamma_{x_{i}}(y)|\le C\left(\delta_{i}^{3}+\tau_{i}\right)
\]
for $|y|\le R/\delta_{i}$.
\end{lem}
\begin{proof}
Let $y_{i}$ such that 
\[
\mu_{i}:=\max_{|y|\le R/\delta_{i}}|v_{i}(y)-U(y)-\delta_{i}^{2}\gamma_{x_{i}}(y)|=|v_{i}(y_{i})-U(y_{i})-\delta_{i}^{2}\gamma_{x_{i}}(y_{i})|.
\]
We can assume, without loss of generality, that $|y_{i}|\le\frac{R}{2\delta_{i}}.$

In fact, suppose that there exists $c>0$ such that $|y_{i}|>\frac{c}{\delta_{i}}$
for all $i$. Then, since $v_{i}(y)\le CU(y)$, and by (\ref{eq:gradvq}),
we get the inequality
\[
|v_{i}(y_{i})-U(y_{i})-\delta_{i}^{2}\gamma_{x_{i}}(y_{i})|\le C\left(|y_{i}|^{2-n}+\delta_{i}^{2}|y_{i}|^{4-n}\right)\le C\delta_{i}^{n-2}
\]
which proves the Lemma. So, in the next we will suppose $|y_{i}|\le\frac{R}{2\delta_{i}}$.
This condition will be exploited later.

To achieve the proof we proceed by contradiction, supposing that 
\begin{equation}
\max\left\{ \mu_{i}^{-1}\delta_{i}^{3},\mu_{i}^{-1}\tau_{i}\right\} \rightarrow0\text{ when }i\rightarrow\infty.\label{eq:ipass}
\end{equation}
Defined 
\[
w_{i}(y):=\mu_{i}^{-1}\left(v_{i}(y)-U(y)-\delta_{i}^{2}\gamma_{x_{i}}(y)\right)\text{ for }|y|\le R/\delta_{i},
\]
we have, by direct computation, that $w_{i}$ satisfies 
\begin{equation}
\left\{ \begin{array}{cc}
L_{\hat{g}_{i}}w_{i}=Q_{i} & \text{ in }B_{\frac{R}{\delta_{i}}}^{+}(0)\\
B_{\hat{g}_{i}}w_{i}+b_{i}w_{i}=\bar{Q}_{i} & \text{ on }\partial'B_{\frac{R}{\delta_{i}}}^{+}(0)
\end{array}\right.\label{eq:wi}
\end{equation}
where 
\begin{align*}
b_{i}= & (n-2)\hat{f}_{i}^{-\tau_{i}}\frac{v_{i}^{p_{i}}-(U+\delta_{i}^{2}\gamma_{x_{i}})^{p_{i}}}{v_{i}-U-\delta_{i}^{2}\gamma_{x_{i}}}\\
\bar{Q}_{i}= & -\frac{1}{\mu_{i}}\left\{ (n-2)(U+\delta_{i}^{2}\gamma_{x_{i}})^{\frac{n}{n-2}}-(n-2)U^{\frac{n}{n-2}}-n\delta_{i}^{2}U^{\frac{2}{n-2}}\gamma_{x_{i}}-\frac{n-2}{2}h_{\hat{g}_{i}}(U+\delta_{i}^{2}\gamma_{x_{i}})\right\} \\
 & +\frac{n-2}{\mu_{i}}\left\{ (U+\delta_{i}^{2}\gamma_{x_{i}})^{\frac{n}{n-2}}-\hat{f}_{i}^{-\tau_{i}}(U+\delta_{i}^{2}\gamma_{x_{i}})^{p_{i}}\right\} =:\bar{Q}_{i,1}+\bar{Q}_{i,2}\\
Q_{i}= & -\frac{1}{\mu_{i}}\left\{ \left(L_{\hat{g}_{i}}-\Delta\right)(U+\delta_{i}^{2}\gamma_{x_{i}})+\delta_{i}^{2}\Delta\gamma_{x_{i}}\right\} .
\end{align*}
We give now some estimate for the terms $b_{i},Q_{i,}\bar{Q}_{i}$
in order to show that the sequence $w_{i}$ converges in $C_{\text{loc}}^{2}(\mathbb{R}_{+}^{n})$
to some $w$ solution of 
\begin{equation}
\left\{ \begin{array}{cc}
\Delta w=0 & \text{ in }\mathbb{R}_{+}^{n}\\
\frac{\partial}{\partial\nu}w+nU^{\frac{n}{n-2}}w=0 & \text{ on }\partial\mathbb{R}_{+}^{n}
\end{array}\right..\label{eq:diff-w}
\end{equation}
Then we will derive a contradiction using (\ref{eq:ipass}). 

By Lagrange Theorem we have 
\[
b_{i}=(n-2)p_{i}\hat{f}_{i}^{-\tau_{i}}\left[\theta v_{i}^{p_{i}-1}-(1-\theta)(U+\delta_{i}^{2}\gamma_{x_{i}})^{p_{i}-1}\right]
\]
and, since $v_{i}\rightarrow U$ in $C_{\text{loc}}^{2}(\mathbb{R}_{+}^{n})$,
we have, at once, 
\begin{align}
b_{i} & \rightarrow nU^{\frac{2}{n-2}}\text{ in }C_{\text{loc}}^{2}(\mathbb{R}_{+}^{n});\label{eq:b1}\\
|b_{i}(y)| & \le(1+|y|)^{-2}\text{ for }|y|\le R/\delta_{i}.\label{eq:b2}
\end{align}
We proceed now by estimating $Q_{i}$ and $\bar{Q}_{i}$. We recall
that 
\begin{align}
[L_{\hat{g}_{i}}-\Delta]u(y)= & \left(\hat{g}_{i}^{kl}-\delta^{kl}\right)\partial_{k}\partial_{l}u+\partial_{k}\hat{g}_{i}^{kl}\partial_{l}u-\frac{n-2}{4(n-1)}R_{\hat{g}_{i}}u\nonumber \\
 & +\frac{\partial_{k}|\hat{g}_{i}|^{\frac{1}{2}}}{|\hat{g}_{i}|^{\frac{1}{2}}}\hat{g}_{i}^{kl}\partial_{l}u\nonumber \\
= & \left(g_{i}^{kl}(\delta_{i}y)-\delta^{kl}\right)\partial_{k}\partial_{l}u+\delta_{i}\partial_{k}g_{i}^{kl}(\delta_{i}y)\partial_{l}u-\delta_{i}^{2}\frac{n-2}{4(n-1)}R_{g_{i}}(\delta_{i}y)u\nonumber \\
 & +O(\delta_{i}^{N}|y|^{N-1})\partial_{l}u\label{eq:L-Delta}
\end{align}
where $N$ can be chosen large since we use conformal Fermi coordinates.
At this point we use the definition of the function $\gamma_{x_{i}}$
(see (\ref{eq:vqdef})), and, by (\ref{eq:L-Delta}), (\ref{eq:gij})
and the decays properties of $U$ and $\gamma_{x_{i}}$, we obtain
\begin{align}
-\mu_{i}Q_{i}= & \delta_{i}^{2}\left(\frac{1}{3}\bar{R}_{kslj}y_{s}y_{j}+R_{nkns}y_{n}^{2}\right)\left(\partial_{k}\partial_{l}U+\delta_{i}^{2}\partial_{k}\partial_{l}\gamma_{x_{i}}\right)\nonumber \\
 & +O(\delta_{i}^{3}|y|^{3})\left(\partial_{k}\partial_{l}U+\delta_{i}^{2}\partial_{k}\partial_{l}\gamma_{x_{i}}\right)\nonumber \\
 & +\delta_{i}^{2}\left(\frac{1}{3}\bar{R}_{kklj}y_{j}+\frac{1}{3}\bar{R}_{kslk}y_{s}\right)\left(\partial_{l}U+\delta_{i}^{2}\partial_{l}\gamma_{x_{i}}\right)\nonumber \\
 & +O(\delta_{i}^{3}|y|^{2})\left(\partial_{l}U+\delta_{i}^{2}\partial_{l}\gamma_{x_{i}}\right)\nonumber \\
 & +O(\delta_{i}^{4}|y|^{2})\left(U+\delta_{i}^{2}\gamma_{x_{i}}\right)\nonumber \\
 & +\delta_{i}^{2}\Delta\gamma_{x_{i}}+O(\delta_{i}^{N}|y|^{N-1})\left(\partial_{l}U+\delta_{i}^{2}\partial_{l}\gamma_{x_{i}}\right)\nonumber \\
= & O\left(\delta_{i}^{3}\left(1+|y|\right)^{3-n}\right)+O\left(\delta_{i}^{4}\left(1+|y|\right)^{4-n}\right)+O\left(\delta_{i}^{5}\left(1+|y|\right)^{5-n}\right)\nonumber \\
 & +O\left(\delta_{i}^{6}\left(1+|y|\right)^{6-n}\right)+O\left(\delta_{i}^{N}\left(1+|y|\right)^{N-n}\right)O\left(\delta_{i}^{N+2}\left(1+|y|\right)^{N+2-n}\right).\label{eq:Qi-parziale}
\end{align}
Since $|y|\le R/\delta_{i}$, we have $\delta_{i}\left(1+|y|\right)\le C$,
thus 
\begin{equation}
Q_{i}=O(\mu_{i}^{-1}\delta_{i}^{3}\left(1+|y|\right)^{3-n}).\label{eq:Q}
\end{equation}
In light of (\ref{eq:ipass}) we have also $Q_{i}\in L^{p}(B_{R/\delta_{i}}^{+})$
for all $p\ge2$.

By Taylor expansion, and proceeding as above, we have
\begin{align*}
-\mu_{i}\bar{Q}_{i,1}= & \left\{ \delta_{i}^{4}\frac{2}{n-2}(U+\theta\delta_{i}^{2}\gamma_{x_{i}})^{\frac{4-n}{n-2}}\gamma_{x_{i}}^{2}-\delta_{i}\frac{n-2}{2}h_{g_{i}}(\delta_{i}y)(U+\delta_{i}^{2}\gamma_{x_{i}})\right\} \\
= & O(\delta_{i}^{4}\left(1+|y|\right)^{5-n}).
\end{align*}
Notice that in the above estimates we have $U+\theta\delta_{i}^{2}\gamma_{x_{i}}>0$
since we are in $B_{R/\delta_{i}}$. 

Similarly, since $(U+\delta_{i}^{2}\gamma_{x_{i}})^{p_{i}}=(U+\delta_{i}^{2}\gamma_{x_{i}})^{\frac{n}{n-2}}+O(\tau_{i})(U+\delta_{i}^{2}\gamma_{x_{i}})^{\frac{n}{n-2}}\log(U+\delta_{i}^{2}\gamma_{x_{i}})$
and $\hat{f}^{-\tau_{i}}=1+O(\tau_{i})$, we have 
\[
-\mu_{i}\bar{Q}_{i,2}=O(\tau_{i}\left(1+|y|\right)^{1-n}).
\]
We conclude 
\begin{equation}
\bar{Q}_{i}=O(\mu_{i}^{-1}\delta_{i}^{4}\left(1+|y|\right)^{5-n})+O(\mu_{i}^{-1}\tau_{i}\left(1+|y|\right)^{1-n}),\label{eq:Qbar}
\end{equation}
and $\bar{Q_{i}}\in L^{p}(\partial'B_{R/\delta_{i}}^{+})$ for all
$p\ge2$. 

Finally we remark that $|w_{i}(y)|\le w_{i}(y_{i})=1$, so by (\ref{eq:ipass}),
(\ref{eq:b1}), (\ref{eq:b2}), (\ref{eq:Q}), (\ref{eq:Qbar}) and
by standard elliptic estimates we conclude that, up to subsequence,
$\left\{ w_{i}\right\} _{i}$ converges in $C_{\text{loc}}^{2}(\mathbb{R}_{+}^{n})$
to some $w$ solution of (\ref{eq:diff-w}), as claimed.

The next step is to prove that $|w(y)|\le C(1+|y|^{-1})$ for $y\in\mathbb{R}_{+}^{n}$.
To do so, we consider $G_{i}$ the Green function for the conformal
Laplacian $L_{\hat{g}_{i}}$ defined on $B_{r/\delta_{i}}^{+}$ with
boundary conditions $B_{\hat{g}_{i}}G_{i}=0$ on $\partial'B_{r/\delta_{i}}^{+}$
and $G_{i}=0$ on $\partial^{+}B_{r/\delta_{i}}^{+}$. It is well
known that $G_{i}=O(|\xi-y|^{2-n})$. By the Green formula and by
(\ref{eq:Q}) and (\ref{eq:Qbar}) we will be able to estimate $w_{i}$
in $B_{R/(2\delta_{i})}^{+}$. In fact 
\begin{align*}
w_{i}(y)= & -\int_{B_{\frac{R}{\delta_{i}}}^{+}}G_{i}(\xi,y)Q_{i}(\xi)d\mu_{\hat{g}_{i}}(\xi)-\int_{\partial^{+}B_{\frac{R}{\delta_{i}}}^{+}}\frac{\partial G_{i}}{\partial\nu}(\xi,y)w_{i}(\xi)d\sigma_{\hat{g}_{i}}(\xi)\\
 & +\int_{\partial'B_{\frac{R}{\delta_{i}}}^{+}}G_{i}(\xi,y)\left(b_{i}(\xi)w_{i}(\xi)-\bar{Q}_{i}(\xi)\right)d\sigma_{\hat{g}_{i}}(\xi),
\end{align*}
so 
\begin{align*}
|w_{i}(y)| & \le\frac{\delta_{i}^{3}}{\mu_{i}}\int_{B_{\frac{R}{\delta_{i}}}^{+}}|\xi-y|^{2-n}(1+|\xi|)^{3-n}d\xi+\int_{\partial^{+}B_{\frac{R}{\delta_{i}}}^{+}}|\xi-y|^{1-n}w_{i}(\xi)d\sigma(\xi)\\
 & +\int_{\partial'B_{\frac{R}{\delta_{i}}}^{+}}|\bar{\xi}-y|^{2-n}\left((1+|\bar{\xi}|)^{-2}+\frac{\delta_{i}^{4}}{\mu_{i}}(1+|\bar{\xi}|)^{5-n}+\frac{\tau_{i}}{\mu_{i}}(1+|\bar{\xi}|)^{1-n}\right)d\bar{\xi},
\end{align*}
where in the last integral we used that $|w_{i}(y)|\le1$. For the
second integral we use that $|y|\le\frac{R}{2\delta_{i}}$ to estimate
$|\xi-y|\ge|\xi|-|y|\ge\frac{R}{2\delta_{i}}$ on $\partial^{+}B_{R/\delta_{i}}^{+}$.
Moreover, since $v_{i}(\xi)\le CU(\xi)$, we get the inequality
\begin{equation}
|w_{i}(\xi)|\le\frac{C}{\mu_{i}}\left(\left(1+|\xi|\right)^{2-n}+\delta_{i}^{2}\left(1+|\xi|\right)^{4-n}\right)\le C\frac{\delta_{i}^{n-2}}{\mu_{i}}\text{ on }\partial^{+}B_{R/\delta_{i}}^{+};\label{eq:wibordo}
\end{equation}
hence 
\begin{equation}
\int_{\partial^{+}B_{\frac{R}{\delta_{i}}}^{+}}|\xi-y|^{1-n}w_{i}(\xi)d\sigma(\xi)\le C\int_{\partial^{+}B_{\frac{R}{\delta_{i}}}^{+}}\frac{\delta_{i}^{2n-3}}{\mu_{i}}d\sigma_{\hat{g}_{i}}(\xi)\le C\frac{\delta_{i}^{n-2}}{\mu_{i}}.\label{eq:stimaW1}
\end{equation}
 For the other terms we use the following formula (see \cite[Lemma 9.2]{Al}
and \cite{Au,Gi}) 
\begin{equation}
\int_{\mathbb{R}^{m}}|\xi-y|^{\beta-m}(1+|y|)^{-\alpha}\le C(1+|y|)^{\beta-\alpha},\label{eq:ALstimagreen}
\end{equation}
which holds for $y\in\mathbb{R}^{m+k}\supseteq\mathbb{R}^{m}$ and
for $\alpha,\beta\in\mathbb{N}$, $0<\beta<\alpha<m$, to obtain 
\begin{equation}
\frac{\delta_{i}^{3}}{\mu_{i}}\int_{B_{\frac{R}{\delta_{i}}}^{+}}|\xi-y|^{2-n}(1+|\xi|)^{3-n}d\xi\le C\frac{\delta_{i}^{3}}{\mu_{i}}(1+|y|)^{5-n};\label{eq:stimaW2}
\end{equation}

\begin{equation}
\int_{\partial'B_{\frac{R}{\delta_{i}}}^{+}}|\bar{\xi}-y|^{2-n}(1+|\bar{\xi}|)^{-2}d\bar{\xi}\le(1+|y|)^{-1};\label{eq:stimaW3}
\end{equation}
\begin{equation}
\frac{\delta_{i}^{4}}{\mu_{i}}\int_{\partial'B_{\frac{R}{\delta_{i}}}^{+}}|\bar{\xi}-y|^{2-n}(1+|\bar{\xi}|)^{5-n}d\bar{\xi}\le C\frac{\delta_{i}^{4}}{\mu_{i}}(1+|y|)^{6-n};\label{eq:stimaW4}
\end{equation}

\begin{equation}
\frac{\tau_{i}}{\mu_{i}}\int_{\partial'B_{\frac{R}{\delta_{i}}}^{+}}|\bar{\xi}-y|^{2-n}(1+|\bar{\xi}|)^{1-n}d\bar{\xi}\le C\frac{\tau_{i}}{\mu_{i}}(1+|y|)^{2-n}.\label{eq:stimaW5}
\end{equation}
By (\ref{eq:stimaW1}), (\ref{eq:stimaW2}), (\ref{eq:stimaW3}),
(\ref{eq:stimaW4}), (\ref{eq:stimaW5}), we have 
\begin{equation}
|w_{i}(y)|\le C\left((1+|y|)^{-1}+\frac{\delta_{i}^{3}}{\mu_{i}}(1+|y|)^{5-n}+\frac{\tau_{i}}{\mu_{i}}(1+|y|)^{2-n}\right)\text{ for }|y|\le\frac{R}{2\delta_{i}},\label{eq:stimaWiass}
\end{equation}
so by assumption (\ref{eq:ipass}) we prove 
\begin{equation}
|w(y)|\le C(1+|y|)^{-1}\text{ for }y\in\mathbb{R}_{+}^{n}\label{eq:stimaWass}
\end{equation}
as claimed.

Finally we notice that, since $v_{i}\rightarrow U$ near $0$, and
by (\ref{eq:dervq}) we have $w_{i}(0)\rightarrow0$ as well as $\frac{\partial w_{i}}{\partial y_{j}}(0)\rightarrow0$
for $j=1,\dots,n-1$. This implies, since $w_{i}\rightarrow w$ in
$C_{\text{loc}}^{2}$, that 
\begin{equation}
w(0)=\frac{\partial w}{\partial y_{1}}(0)=\dots=\frac{\partial w}{\partial y_{n-1}}(0)=0.\label{eq:W(0)}
\end{equation}
We are ready now to prove the contradiction. In fact, it is known
(see \cite[Lemma 2]{Al}) that any solution of (\ref{eq:diff-w})
that decays as (\ref{eq:stimaWass}) is a linear combination of $\frac{\partial U}{\partial y_{1}},\dots,\frac{\partial U}{\partial y_{n-1}},\frac{n-2}{2}U+y^{b}\frac{\partial U}{\partial y_{b}}$.
This fact, combined with (\ref{eq:W(0)}), implies that $w\equiv0$. 

Now, on one hand $|y_{i}|\le\frac{R}{2\delta_{i}}$, so estimate (\ref{eq:stimaWiass})
holds; on the other hand, since $w_{i}(y_{i})=1$ and $w\equiv0$,
we get $|y_{i}|\rightarrow\infty$, obtaining
\[
1=w_{i}(y_{i})\le C(1+|y_{i}|)^{-1}\rightarrow0
\]
 which gives us the desired contradiction, and proves the Lemma. 
\end{proof}
\begin{lem}
\label{lem:taui}Assume $n\ge8$. There exists $C>0$ such that 
\[
\tau_{i}\le C\delta_{i}^{3}.
\]
\end{lem}
\begin{proof}
We proceed by contradiction, supposing that 
\begin{equation}
\tau_{i}^{-1}\delta_{i}^{3}\rightarrow0\text{ when }i\rightarrow\infty.\label{eq:ipasstau}
\end{equation}
Thus, by Lemma \ref{lem:coreLemma}, we have 
\[
|v_{i}(y)-U(y)-\delta_{i}^{2}\gamma_{x_{i}}(y)|\le C\tau_{i}\text{ for }|y|\le R/\delta_{i}.
\]
We define, similarly to Lemma \ref{lem:coreLemma},
\[
w_{i}(y):=\frac{1}{\tau_{i}}\left(v_{i}(y)-U(y)-\delta_{i}^{2}\gamma_{x_{i}}(y)\right)\text{ for }|y|\le R/\delta_{i},
\]
and we have that $w_{i}$ satisfies (\ref{eq:wi}), where 
\begin{align*}
b_{i}= & (n-2)\hat{f_{i}}^{-\tau_{i}}\frac{v_{i}^{p_{i}}-(U+\delta_{i}^{2}\gamma_{x_{i}})^{p_{i}}}{v_{i}-U-\delta_{i}^{2}\gamma_{x_{i}}}\\
\bar{Q}_{i}= & -\frac{1}{\tau_{i}}\left\{ (n-2)(U+\delta_{i}^{2}\gamma_{x_{i}})^{\frac{n}{n-2}}-(n-2)U^{\frac{n}{n-2}}-n\delta_{i}^{2}U^{\frac{2}{n-2}}\gamma_{x_{i}}-\frac{n-2}{2}h_{\hat{g}_{i}}(U+\delta_{i}^{2}\gamma_{x_{i}})\right\} \\
 & +\frac{n-2}{\tau_{i}}\left\{ (U+\delta_{i}^{2}\gamma_{x_{i}})^{\frac{n}{n-2}}-\hat{f}_{i}^{-\tau_{i}}(U+\delta_{i}^{2}\gamma_{x_{i}})^{p_{i}}\right\} \\
Q_{i}= & -\frac{1}{\tau_{i}}\left\{ \left(L_{\hat{g}_{i}}-\Delta\right)(U+\delta_{i}^{2}\gamma_{x_{i}})+\delta_{i}^{2}\Delta\gamma_{x_{i}}\right\} .
\end{align*}
As before, $b_{i}$ satisfies inequality (\ref{eq:b2}), while 
\begin{align}
Q_{i} & =O(\tau_{i}^{-1}\delta_{i}^{3}\left(1+|y|\right)^{3-n}),\label{eq:Q-1-1}\\
\bar{Q}_{i} & =O(\tau_{i}^{-1}\delta_{i}^{4}\left(1+|y|\right)^{5-n})+O(\left(1+|y|\right)^{1-n}),\label{eq:Q-1-2}
\end{align}
and we can proceed as in Lemma \ref{lem:coreLemma}, to deduce that
\begin{equation}
|w_{i}(y)|\le C\left((1+|y|)^{-1}+\frac{\delta_{i}^{3}}{\tau_{i}}(1+|y|)^{5-n}\right)\text{ for }|y|\le\frac{R}{2\delta_{i}}.\label{eq:decWtau}
\end{equation}
By classic elliptic estimates, we can prove that the sequence $w_{i}$
converges in $C_{\text{loc}}^{2}(\mathbb{R}_{+}^{n})$ to some $w$.

Finally, by assumption on $\left\{ f_{i}\right\} _{i}$, $f_{i}\rightarrow\Lambda_{x_{0}}f$
in the $C^{1}$ topology, and since $\hat{f}_{i}(y)=f(\delta_{i}y),$
and recalling that $x_{i}=\psi_{i}(0)$, $x_{i}\rightarrow x_{0}$
and $\Lambda_{x_{0}}(x_{0})=1$, we have 
\begin{equation}
\lim_{i\rightarrow+\infty}\frac{1}{\tau_{i}}\left\{ (U+\delta_{i}^{2}\gamma_{x_{i}})^{\frac{n}{n-2}}-\hat{f}_{i}^{-\tau_{i}}(U+\delta_{i}^{2}\gamma_{x_{i}})^{p_{i}}\right\} =\left[\log(f(x_{0}))+\log U\right]U^{\frac{n}{n-2}}.\label{eq:limQ2}
\end{equation}
Now, let $j_{n}$ defined as in (\ref{eq:jn}). Since $\int_{\mathbb{R}_{+}^{n}}j_{n}(y)U^{\frac{n}{n-2}}(y)dy=0$,
and in light of (\ref{eq:limQ2}) and (\ref{eq:Q-1-2}), we get 
\[
\lim_{i\rightarrow+\infty}\int_{\partial'B_{\frac{R}{\delta_{i}}}^{+}}j_{n}\bar{Q}_{i}d\sigma_{\hat{g}_{i}}=(n-2)\int_{\partial\mathbb{R}_{+}^{n}}j_{n}(y)\log U(y)U^{\frac{n}{n-2}}(y)dy.
\]
By direct computation we have 
\begin{equation}
(n-2)\int_{\partial\mathbb{R}_{+}^{n}}j_{n}(y)\log U(y)U^{\frac{n}{n-2}}(y)dy>0.\label{eq:neg}
\end{equation}
In fact, integrating in polar coordinates $r:=|\bar{y}|$ on $\partial\mathbb{R}_{+}^{n}$,
we obtain
\begin{align*}
\int_{\partial\mathbb{R}_{+}^{n}}j_{n}(y)\log U(y)U^{\frac{n}{n-2}}(y)dy & =-\sigma_{n-2}\frac{(n-2)^{2}}{4}\int_{0}^{\infty}\frac{1-r^{2}}{(1+r^{2})^{n}}r^{n-2}\log(1+r^{2})dr>0.
\end{align*}
At this point we can see that (\ref{eq:neg}) leads us to a contradiction.
Indeed, since $w_{i}$ satisfies (\ref{eq:wi}), integrating by parts
we obtain
\begin{multline*}
\int_{\partial'B_{\frac{R}{\delta_{i}}}^{+}}j_{n}\bar{Q}_{i}d\sigma_{\hat{g}_{i}}=\int_{\partial'B_{\frac{R}{\delta_{i}}}^{+}}j_{n}\left[B_{\hat{g}_{i}}w_{i}+b_{i}w_{i}\right]d\sigma_{\hat{g}_{i}}\\
=\int_{\partial'B_{\frac{R}{\delta_{i}}}^{+}}w_{i}\left[B_{\hat{g}_{i}}j_{n}+b_{i}j_{n}\right]d\sigma_{\hat{g}_{i}}+\int_{\partial^{+}B_{\frac{R}{\delta_{i}}}^{+}}\left[\frac{\partial j_{n}}{\partial\eta_{i}}w_{i}-\frac{\partial w_{i}}{\partial\eta_{i}}j_{n}\right]d\sigma_{\hat{g}_{i}}\\
+\int_{B_{\frac{R}{\delta_{i}}}^{+}}\left[w_{i}L_{\hat{g}_{i}}j_{n}-j_{n}L_{\hat{g}_{i}}w_{i}\right]d\mu_{\hat{g}_{i}},
\end{multline*}
where $\eta_{i}$ is the inward unit normal vector to $\partial^{+}B_{\frac{R}{\delta_{i}}}^{+}$. 

By the decay of $j_{n}$ and by the decay of $w_{i}$, given by (\ref{eq:decWtau})
and by (\ref{eq:ipasstau}), we have
\begin{equation}
\lim_{i\rightarrow+\infty}\int_{\partial^{+}B_{\frac{R}{\delta_{i}}}^{+}}\left[\frac{\partial j_{n}}{\partial\eta_{i}}w_{i}-\frac{\partial w_{i}}{\partial\eta_{i}}j_{n}\right]d\sigma_{\hat{g}_{i}}=0\label{eq:nullo1}
\end{equation}
and by (\ref{eq:wi}) and by the decay of $Q_{i}$ given in (\ref{eq:Q-1-1})
we have 
\begin{equation}
\lim_{i\rightarrow+\infty}\int_{B_{\frac{R}{\delta_{i}}}^{+}}j_{n}L_{\hat{g}_{i}}w_{i}d\mu_{\hat{g}_{i}}=\int_{B_{\frac{R}{\delta_{i}}}^{+}}j_{n}Q_{i}d\mu_{\hat{g}_{i}}=0.\label{eq:nullo2}
\end{equation}
Finally, since $\Delta j_{n}=0$, by (\ref{eq:L-Delta}) we get 
\begin{equation}
\lim_{i\rightarrow+\infty}\int_{B_{\frac{R}{\delta_{i}}}^{+}}w_{i}L_{\hat{g}_{i}}j_{n}d\mu_{\hat{g}_{i}}=0,\label{eq:nullo3}
\end{equation}
thus, by (\ref{eq:nullo1}), (\ref{eq:nullo2}) and (\ref{eq:nullo3}),
we have
\begin{align}
\lim_{i\rightarrow+\infty}\int_{\partial'B_{\frac{R}{\delta_{i}}}^{+}}j_{n}\bar{Q}_{i}d\sigma_{\hat{g}_{i}} & =\lim_{i\rightarrow+\infty}\int_{\partial'B_{\frac{R}{\delta_{i}}}^{+}}w_{i}\left[B_{\hat{g}_{i}}j_{n}+b_{i}j_{n}\right]d\sigma_{\hat{g}_{i}}\nonumber \\
 & =\int_{\partial\mathbb{R}_{+}^{n}}w\left[\frac{\partial j_{n}}{\partial y_{n}}+nU^{\frac{2}{n-2}}j_{n}\right]d\sigma_{\hat{g}_{i}}=0\label{eq:nullofinale}
\end{align}
since $\frac{\partial j_{n}}{\partial y_{n}}+nU^{\frac{2}{n-2}}j_{n}=0$
when $y_{n}=0$. Comparing (\ref{eq:neg}) and (\ref{eq:nullofinale})
we get the contradiction.
\end{proof}
The above lemmas are the core of the following proposition, in which
we iterate the procedure of Lemma \ref{lem:coreLemma}, to obtain
better estimates of the rescaled solution $v_{i}$ of (\ref{eq:Prob-hat})
around the isolated simple blow up point $x_{i}\rightarrow x_{0}$.
\begin{prop}
\label{prop:stimawi}Assume $n\ge8$. Let $\gamma_{x_{i}}$ be defined
in (\ref{eq:vqdef}). There exist $R,C>0$ such that 
\begin{align*}
|v_{i}(y)-U(y)-\delta_{i}^{2}\gamma_{x_{i}}(y)| & \le C\delta_{i}^{3}(1+|y|)^{5-n}\\
\left|\frac{\partial}{\partial_{j}}\left(v_{i}(y)-U(y)-\delta_{i}^{2}\gamma_{x_{i}}(y)\right)\right| & \le C\delta_{i}^{3}(1+|y|)^{4-n}\\
\left|y_{n}\frac{\partial}{\partial_{n}}\left(v_{i}(y)-U(y)-\delta_{i}^{2}\gamma_{x_{i}}(y)\right)\right| & \le C\delta_{i}^{3}(1+|y|)^{5-n}\\
\left|\frac{\partial^{2}}{\partial_{j}\partial_{k}}\left(v_{i}(y)-U(y)-\delta_{i}^{2}\gamma_{x_{i}}(y)\right)\right| & \le C\delta_{i}^{3}(1+|y|)^{3-n}
\end{align*}
for $|y|\le\frac{R}{2\delta_{i}}$. Here $j,k=1,\dots,n-1$.
\end{prop}
\begin{proof}
In analogy with Lemma \ref{lem:coreLemma}, we set 
\[
w_{i}(y):=v_{i}(y)-U(y)-\delta_{i}^{2}\gamma_{x_{i}}(y)\text{ for }|y|\le R/\delta_{i},
\]
and we have that $w_{i}$ satisfies (\ref{eq:wi}), where 
\begin{align*}
b_{i}= & (n-2)\hat{f}^{-\tau_{i}}\frac{v_{i}^{p_{i}}-(U+\delta_{i}^{2}\gamma_{x_{i}})^{p_{i}}}{v_{i}-U-\delta_{i}^{2}\gamma_{x_{i}}}\\
\bar{Q}_{i}= & -\left\{ (n-2)(U+\delta_{i}^{2}\gamma_{x_{i}})^{\frac{n}{n-2}}-(n-2)U^{\frac{n}{n-2}}-n\delta_{i}^{2}U^{\frac{2}{n-2}}\gamma_{x_{i}}-\frac{n-2}{2}h_{\hat{g}_{i}}(U+\delta_{i}^{2}\gamma_{x_{i}})\right\} \\
 & +\frac{n-2}{\tau_{i}}\left\{ (U+\delta_{i}^{2}\gamma_{x_{i}})^{\frac{n}{n-2}}-\hat{f}^{-\tau_{i}}(U+\delta_{i}^{2}\gamma_{x_{i}})^{p_{i}}\right\} \\
Q_{i}= & -\left\{ \left(L_{\hat{g}_{i}}-\Delta\right)(U+\delta_{i}^{2}\gamma_{x_{i}})+\delta_{i}^{2}\Delta\gamma_{x_{i}}\right\} .
\end{align*}
As before, $b_{i}$ satisfies inequality (\ref{eq:b2}) and 
\begin{align}
Q_{i} & =O(\delta_{i}^{3}\left(1+|y|\right)^{3-n})\label{eq:Q-1}\\
\bar{Q}_{i} & =O(\delta_{i}^{4}\left(1+|y|\right)^{5-n})+O(\delta_{i}^{3}\left(1+|y|\right)^{1-n})=O(\delta_{i}^{3}\left(1+|y|\right)^{5-n}).\label{eq:Q-2}
\end{align}
We define again the Green function $G_{i}$ as in the previous lemma
and we have, by Green formula,
\begin{align}
|w_{i}(y)|\le & \int_{B_{\frac{R}{\delta_{i}}}^{+}}|\xi-y|^{2-n}Q_{i}(\xi)d\xi+\int_{\partial^{+}B_{\frac{R}{\delta_{i}}}^{+}}|\xi-y|^{1-n}w_{i}(\xi)d\sigma(\xi)\nonumber \\
 & +\int_{\partial'B_{\frac{R}{\delta_{i}}}^{+}}|\bar{\xi}-y|^{2-n}b_{i}(\xi)w_{i}(\xi)d\bar{\xi})+\int_{\partial'B_{\frac{R}{\delta_{i}}}^{+}}|\bar{\xi}-y|^{2-n}\bar{Q}_{i}(\xi)d\bar{\xi}.\label{eq:Green}
\end{align}
By the results of Lemma \ref{lem:coreLemma} and Lemma \ref{lem:taui},
and in analogy with equation (\ref{eq:wibordo}), we have that 
\begin{eqnarray}
|w_{i}(y)|\le C\delta_{i}^{3}\text{ on }B_{R/\delta_{i}}^{+} & \text{ and } & |w_{i}(\xi)|\le C\delta_{i}^{n-2}\text{ on }\partial^{+}B_{R/\delta_{i}}^{+}.\label{eq:w_i-improved}
\end{eqnarray}
Plugging (\ref{eq:b2}), (\ref{eq:Q-1}), (\ref{eq:Q-2}) and (\ref{eq:w_i-improved})
in (\ref{eq:Green}) and proceeding as in Lemma \ref{lem:coreLemma}
we obtain
\begin{align}
\int_{B_{\frac{R}{\delta_{i}}}^{+}}|\xi-y|^{2-n}Q_{i}(\xi)d\xi & \le C\delta_{i}^{3}(1+|y|)^{5-n}\label{eq:W1-1}\\
\int_{\partial^{+}B_{\frac{R}{\delta_{i}}}^{+}}|\xi-y|^{1-n}w_{i}(\xi)d\sigma(\xi) & \le C\delta_{i}^{n-2}\label{eq:W2-1}\\
\int_{\partial'B_{\frac{R}{\delta_{i}}}^{+}}|\bar{\xi}-y|^{2-n}b_{i}(\xi)w_{i}(\xi)d\bar{\xi}) & \le\delta_{i}^{3}(1+|y|)^{-1}\label{eq:W3-1}\\
\int_{\partial'B_{\frac{R}{\delta_{i}}}^{+}}|\bar{\xi}-y|^{2-n}\bar{Q}_{i}(\xi)d\bar{\xi} & \le C\delta_{i}^{3}(1+|y|)^{5-n}\label{eq:W4-1}
\end{align}
for $|y|\le\frac{R}{2\delta_{i}},$ which implies 
\begin{equation}
|w_{i}(y)|\le C\delta_{i}^{3}(1+|y|)^{-1}\text{ for }|y|\le\frac{R}{2\delta_{i}}.\label{eq:wi-improv2}
\end{equation}
We now iterate this procedure, inserting inequality (\ref{eq:wi-improv2})
in equation (\ref{eq:Green}). Inequalities (\ref{eq:W1-1}), (\ref{eq:W2-1})
and (\ref{eq:W4-1}) do not improve, while for (\ref{eq:W3-1})
we have
\begin{align}
\int_{\partial'B_{\frac{R}{\delta_{i}}}^{+}}|\bar{\xi}-y|^{2-n}b_{i}(\xi)w_{i}(\xi)d\bar{\xi}) & \le\delta_{i}^{3}(1+|y|)^{-2}\label{eq:W3-2}
\end{align}
for $|y|\le\frac{R}{2\delta_{i}}$, getting
\begin{equation}
|w_{i}(y)|\le\delta_{i}^{3}(1+|y|)^{-2}\text{ for }|y|\le\frac{R}{2\delta_{i}}.\label{eq:wi-improv2-1}
\end{equation}
We iterate again to further improve estimate (\ref{eq:W3-2}), until
we reach 
\begin{equation}
|w_{i}(y)|\le C\delta_{i}^{3}(1+|y|)^{5-n}\text{ for }|y|\le\frac{R}{2\delta_{i}},\label{eq:wi-improv2-1-1}
\end{equation}
which proves the first claim.

To prove the estimate for $y_{n}\frac{\partial}{\partial y_{n}}w_{i}$,
we differentiate the Green formula obtaining
\begin{align*}
y_{n}\frac{\partial}{\partial y_{n}}w_{i}(y)= & -y_{n}\int_{B_{\frac{R}{\delta_{i}}}^{+}}\frac{\partial}{\partial y_{n}}G_{i}(\xi,y)Q_{i}(\xi)d\mu_{\hat{g}_{i}}(\xi)-y_{n}\int_{\partial^{+}B_{\frac{R}{\delta_{i}}}^{+}}\frac{\partial}{\partial y_{n}}\frac{\partial G_{i}}{\partial\nu}(\xi,y)w_{i}(\xi)d\sigma_{\hat{g}_{i}}(\xi)\\
 & +\int_{\partial'B_{\frac{R}{\delta_{i}}}^{+}}y_{n}\frac{\partial}{\partial y_{n}}G_{i}(\xi,y)\left(b_{i}(\xi)w_{i}(\xi)-\bar{Q}_{i}(\xi)\right)d\sigma_{\hat{g}_{i}}(\xi),
\end{align*}
ans since $\frac{\partial}{\partial y_{n}}G_{i}(\xi,y)=O(|\xi-y|^{1-n})$,
we can proceed as above for the first two integrals. Then we use the
trivial estimate $|y_{n}|\le(1+|y|)$ to obtain the desired inequality.
The last term is more delicate, since we cannot use directly estimate
(\ref{eq:ALstimagreen}), for the restriction on the exponents. Anyway,
since $\xi_{n}=0$ on $\partial'B_{\frac{R}{\delta_{i}}}^{+}$, we
have 
\[
\left.\frac{\partial}{\partial y_{n}}G_{i}(\xi,y)\right|_{\partial'B_{\frac{R}{\delta_{i}}}^{+}}=O(|\xi-y|^{-n}y_{n})
\]
and, since $y_{n}^{2}\le|\xi-y|^{2}$ on $\partial'B_{\frac{R}{\delta_{i}}}^{+}$,
we conclude 
\[
\left.y_{n}\frac{\partial}{\partial y_{n}}G_{i}(\xi,y)\right|_{\partial'B_{\frac{R}{\delta_{i}}}^{+}}=O(|\xi-y|^{-n}y_{n}^{2})=O(|\xi-y|^{2-n}).
\]
At this point we have 
\begin{multline*}
\int_{\partial'B_{\frac{R}{\delta_{i}}}^{+}}y_{n}\frac{\partial}{\partial y_{n}}G_{i}(\xi,y)\left(b_{i}(\xi)w_{i}(\xi)-\bar{Q}_{i}(\xi)\right)d\sigma_{\hat{g}_{i}}(\xi)\\
\le C\int_{\partial'B_{\frac{R}{\delta_{i}}}^{+}}|\xi-y|^{2-n}\left(b_{i}(\xi)w_{i}(\xi)-\bar{Q}_{i}(\xi)\right)d\sigma_{\hat{g}_{i}}(\xi)
\end{multline*}
and we are in position to use (\ref{eq:ALstimagreen}). Then we can
obtain the desired estimate with the same technique of Lemma \ref{lem:coreLemma}. 

To prove the estimates for $\frac{\partial}{\partial y_{k}}w_{i}$,
we have to differentiate equation (\ref{eq:wi}), getting for $k=1,\dots n-1$
\[
\left\{ \begin{array}{cc}
L_{\hat{g}_{i}}\frac{\partial}{\partial y_{k}}w_{i}=\frac{\partial}{\partial y_{k}}Q_{i}-\frac{\partial}{\partial y_{k}}R_{g}w_{i} & \text{ in }B_{\frac{R}{\delta_{i}}}^{+}(0)\\
B_{\hat{g}_{i}}\frac{\partial}{\partial y_{k}}w_{i}=\frac{\partial}{\partial y_{k}}\left[\bar{Q}_{i}-b_{i}w_{i}\right]-\left(\frac{\partial}{\partial y_{k}}h_{g}\right)w_{i} & \text{ on }\partial'B_{\frac{R}{\delta_{i}}}^{+}(0)
\end{array}\right.
\]
and we can repeat the strategy contained in Lemma \ref{lem:coreLemma}
and in this proof to obtain the claim. For the estimate on the second
derivatives we proceed analogously.
\end{proof}

\section{\label{sec:Poho}A Pohozaev type identity}

We present here an analogous of the well known Pohozaev identity.
\begin{thm}[Pohozaev Identity]
\label{thm:poho} Let $u$ a $C^{2}$-solution of the following problem
\[
\left\{ \begin{array}{cc}
L_{g}u=0 & \text{ in }B_{r}^{+}\\
B_{g}u+(n-2)f^{-\tau}u^{p}=0 & \text{ on }\partial'B_{r}^{+}
\end{array}\right.
\]
for $B_{r}^{+}=\psi_{q}^{-1}(B_{g}^{+}(q,r))$ for $q\in\partial M$,
with $\tau=\frac{n}{n-2}-p>0$. 
\[
P(u,r):=\int\limits _{\partial^{+}B_{r}^{+}}\left(\frac{n-2}{2}u\frac{\partial u}{\partial r}-\frac{r}{2}|\nabla u|^{2}+r\left|\frac{\partial u}{\partial r}\right|^{2}\right)d\sigma_{r}+\frac{r(n-2)}{p+1}\int\limits _{\partial(\partial'B_{r}^{+})}f^{-\tau}u^{p+1}d\bar{\sigma}_{g}.
\]
Then
\begin{multline*}
P(u,r)=-\int\limits _{B_{r}^{+}}\left(y^{a}\partial_{a}u+\frac{n-2}{2}u\right)[(L_{g}-\Delta)u]dy+\frac{n-2}{2}\int\limits _{\partial'B_{r}^{+}}\left(\bar{y}^{k}\partial_{k}u+\frac{n-2}{2}u\right)h_{g}ud\bar{y}\\
-\frac{\tau(n-2)}{p+1}\int\limits _{\partial'B_{r}^{+}}\left(\bar{y}^{k}\partial_{k}f\right)f^{-\tau-1}u^{p+1}d\bar{y}+\left(\frac{n-1}{p+1}-\frac{n-2}{2}\right)\int\limits _{\partial'B_{r}^{+}}(n-2)f^{-\tau}u^{p+1}d\bar{y}.
\end{multline*}
We recall that $a=1,\dots,n$, $k=1,\dots,n-1$ and $y=(\bar{y},y_{n})$,
where $\bar{y}\in\mathbb{R}^{n-1}$ and $y_{n}\ge0$.
\end{thm}
\begin{proof}
The proof is essentially identical to the classical Pohozaev identity:
we multiply equation by $y^{a}\partial_{a}u$ and we integrate by
parts. All the details can be found in \cite[Prop. 3.1]{Al}.
\end{proof}

\section{\label{sec:Sign}Sign estimates of Pohozaev identity terms}

In this section, we want to estimate $P(u_{i},r)$, where $\left\{ u_{i}\right\} _{i}$
is a family of solutions of (\ref{eq:Prob-i}) which has an isolated
simple blow up point $x_{i}\rightarrow x_{0}$.

Since the leading term of $P(u_{i},r)$ will be $-\int_{B_{r/\delta_{i}}^{+}}\left(y^{b}\partial_{b}u+\frac{n-2}{2}u\right)\left[(L_{\hat{g}_{i}}-\Delta)v\right]dy$
we set
\begin{equation}
R(u,v)=-\int_{B_{r/\delta_{i}}^{+}}\left(y^{b}\partial_{b}u+\frac{n-2}{2}u\right)\left[(L_{\hat{g}_{i}}-\Delta)v\right]dy.\label{eq:Ruv}
\end{equation}

\begin{prop}
\label{prop:segno}Let $x_{i}\rightarrow x_{0}$ be an isolated simple
blow-up point for $u_{i}$ solutions of (\ref{eq:Prob-i}). Then,
fixed $r$, we have 
\begin{align*}
P(u_{i},r)\ge & \delta_{i}^{4}\frac{(n-2)\omega_{n-2}I_{n}^{n}}{(n-1)(n-3)(n-5)(n-6)}\left[\frac{\left(n-2\right)}{6}|\bar{W}(x_{i})|^{2}+\frac{4(n-8)}{(n-4)}R_{nlnj}^{2}(x_{i})\right]\\
 & -2\delta_{i}^{4}\int_{\mathbb{R}_{+}^{n}}\gamma_{x_{i}}\Delta\gamma_{x_{i}}dy+o(\delta_{i}^{4}).
\end{align*}
Here $\displaystyle{I_n^n=\int_0^\infty \frac{s^n}{(1+s^2)^n}ds>0}$.
\end{prop}
\begin{proof}
We have, by Theorem \ref{thm:poho}, and recalling that $\tau_{i}=\frac{n}{n-2}-p_{i}$,
\begin{multline*}
P(u_{i},r)=-\int\limits _{B_{r}^{+}}\left(y^{a}\partial_{a}u_{i}+\frac{n-2}{2}u_{i}\right)[(L_{g_{i}}-\Delta)u_{i}]dy\\
+\frac{n-2}{2}\int\limits _{\partial'B_{r}^{+}}\left(\bar{y}^{k}\partial_{k}u_{i}+\frac{n-2}{2}u_{i}\right)h_{g_{i}}u_{i}d\bar{y}\\
+\frac{\tau_{i}(n-2)}{p_{i}+1}\left[\frac{n-2}{2}\int\limits _{\partial'B_{r}^{+}}f_{i}^{-\tau_{i}}u_{i}^{p_{i}+1}d\bar{y}
-\int\limits _{\partial'B_{r}^{+}}\left(\bar{y}^{k}\partial_{k}f_{i}\right)f_{i}^{-\tau_i-1}u_{i}^{p_{i}+1}d\bar{y}\right].
\end{multline*}
where $B_{r}^{+}$ is the counterimage of $B_{g_{i}}^{+}(x_{i},r)$
by $\psi_{x_{i}}$. Since $f_{i}$ are positive, bounded away from
$0$, and bounded in the $C^{1}$ topology, we can choose $r$ sufficiently
small in order to have 
\[
\frac{\tau_{i}(n-2)}{p_{i}+1}\left[\frac{n-2}{2}\int\limits _{\partial'B_{r}^{+}}f_{i}^{-\tau_{i}}u_{i}^{p_{i}+1}d\bar{y}-\int\limits _{\partial'B_{r}^{+}}\left(\bar{y}^{k}\partial_{k}f_{i}\right)f_{i}^{-\tau-1}u_{i}^{p_{i}+1}d\bar{y}\right]\ge0.
\]
Now, set 
\[
v_{i}(y):=\delta_{i}^{\frac{1}{p_{i}-1}}u_{i}(\delta_{i}y)\text{ for }y\in B_{\frac{R}{\delta_{i}}}^{+}(0).
\]
After a change of variables we obtain
\begin{align*}
P(u_{i},r) & \ge-\delta_{i}^{n-2-\frac{2}{p_{i}-1}}\int_{B_{r/\delta_{i}}^{+}}\left(y^{b}\partial_{b}v_i+\frac{n-2}{2}v_{i}\right)\left[(L_{\hat{g}_{i}}-\Delta)v_{i}\right]dy\\
 & +\frac{n-2}{2}\delta_{i}^{n-2-\frac{2}{p_{i}-1}}\int_{\partial'B_{r}^{+}}\left(y^{b}\partial_{b}v_{i}+\frac{n-2}{2}v_{i}\right)h_{g_{i}}(\delta_{i}y)v_{i}d\bar{y}.
\end{align*}
Since $h_{g_{i}}(\delta_{i}y)=O(\delta_{i}^{4}|y|^{4})$ and $\lim_{i}\delta_{i}^{n-2-\frac{2}{p_{i}-1}}=1$,
we have

\begin{multline*}
\delta_{i}^{n-2-\frac{2}{p_{i}-1}}\int_{\partial'B_{r/\delta_{i}}^{+}}\left(y^{b}\partial_{b}v_{i}+\frac{n-2}{2}v_{i}\right)h_{g_{i}}(\delta_{i}y)v_{i}d\bar{y}\\
=O(\delta_{i}^{5})\int_{\partial'B_{r/\delta_{i}}^{+}}(1+|y|)^{4-2n}|y|^{4}dy=O(\delta_{i}^{5})\text{ for }n\ge8.
\end{multline*}
So 
\[
P(u_{i},r)\ge-\int_{B_{r/\delta_{i}}^{+}}\left(y^{b}\partial_{b}v_{i}+\frac{n-2}{2}v_{i}\right)\left[(L_{\hat{g}_{i}}-\Delta)v_{i}\right]dy+O(\delta_{i}^{5})
\]
Now define, in analogy with Proposition \ref{prop:stimawi}, 
\[
w_{i}(y):=v_{i}(y)-U(y)-\delta_{i}^{2}\gamma_{x_{i}}(y).
\]
Recalling (\ref{eq:Ruv}), we have 
\begin{align*}
P(u_{i},r) & \ge R(U,U)+R(U,\delta_{i}^{2}\gamma_{x_{i}})+R(\delta_{i}^{2}\gamma_{x_{i}},U)+R(w_{i},U)+R(U,w_{i})\\
 & +R(w_{i,}w_{i})+R(\delta_{i}^{2}\gamma_{q},\delta_{i}^{2}\gamma_{x_{i}})+R(w_{i},\delta_{i}^{2}\gamma_{x_{i}})+R(\delta_{i}^{2}\gamma_{x_{i}},w_{i})+O(\delta_{i}^{5})
\end{align*}
and, by the following Lemma \ref{lem:R(U,U)}, Lemma \ref{lem:R(U,vq)},
and Lemma \ref{lem:Raltri}, we conlcude 
\begin{align*}
P(u_{i},r)\ge & R(U,U)+R(U,\delta_{i}^{2}\gamma_{x_{i}})+R(\delta_{i}^{2}\gamma_{x_{i}},U)+o(\delta_{i}^{4})\\
= & \delta_{i}^{4}\frac{(n-2)\omega_{n-2}I_{n}^{n}}{(n-1)(n-3)(n-5)(n-6)}\left[\frac{\left(n-2\right)}{6}|\bar{W}(x_{i})|^{2}+\frac{4(n-8)}{(n-4)}R_{nlnj}^{2}(x_{i})\right]\\
 & -2\delta_{i}^{4}\int_{\mathbb{R}_{+}^{n}}\gamma_{x_{i}}\Delta\gamma_{x_{i}}dy+o(\delta_{i}^{4})
\end{align*}
and we prove the result.
\end{proof}
In order to simplify the notation, in the following lemmas we use
$\delta=\delta_{i}$ and $q=x_{i}$.
\begin{lem}
\label{lem:R(U,U)}We have 
\begin{align*}
R(U,U) & =\delta^{4}\frac{(n-2)\omega_{n-2}I_{n}^{n}}{(n-1)(n-3)(n-5)(n-6)}\left[\frac{\left(n-2\right)}{6}|\bar{W}(q)|^{2}+\frac{4(n-8)}{(n-4)}R_{ninj}^{2}\right]+o(\delta^{4})
\end{align*}
\end{lem}
\begin{proof}
Recalling that $U$ is the standard bubble and equation (\ref{eq:L-Delta}),
we obtain
\begin{align*}
R(U,U)= & \frac{\left(n-2\right)^{2}}{2}\int_{B_{r\delta^{-1}}^{+}}\frac{|y|^{2}-1}{\left[(1+y_{n})^{2}+|\bar{y}|^{2}\right]^{n+1}}ny_{i}y_{j}\left(g^{ij}(\delta y)-\delta^{ij}\right)dy\\
 & -\frac{\left(n-2\right)^{2}}{2}\int_{B_{r\delta^{-1}}^{+}}\frac{|y|^{2}-1}{\left[(1+y_{n})^{2}+|\bar{y}|^{2}\right]^{n}}\left(g^{jj}(\delta y)-1\right)dy\\
 & -\frac{\left(n-2\right)^{2}}{2}\int_{B_{r\delta^{-1}}^{+}}\frac{|y|^{2}-1}{\left[(1+y_{n})^{2}+|\bar{y}|^{2}\right]^{n}}\delta\partial_{i}g^{ij}(\delta y)y_{j}dy\\
 & -\frac{\left(n-2\right)^{2}}{8(n-1)}\int_{B_{r\delta^{-1}}^{+}}\frac{|y|^{2}-1}{\left[(1+y_{n})^{2}+|\bar{y}|^{2}\right]^{n-1}}\delta^{2}R_{g}(\delta y)dy+O(\delta^{5})\\
 & =:A_{1}+A_{2}+A_{3}+A_{4}+O(\delta^{5}).
\end{align*}
For the sake of simplicity we call $L_{1}(y):=\frac{|y|^{2}-1}{\left[(1+y_{n})^{2}+|\bar{y}|^{2}\right]^{n+1}}$,
$L_{2}(y):=\frac{|y|^{2}-1}{\left[(1+y_{n})^{2}+|\bar{y}|^{2}\right]^{n}}$
and $L_{3}(y):=\frac{|y|^{2}-1}{\left[(1+y_{n})^{2}+|\bar{y}|^{2}\right]^{n-1}}$.
By symmetry arguments we have only to consider the fourth order terms
in the expansion of $g^{ij}$. Since $B_{r\delta^{-1}}^{+}$ invades
$\mathbb{R}^{n-1}\times\mathbb{R}^{+}$ as $\delta\rightarrow0^{+}$,
and recalling the expansion of $g^{ij}$ we have

\begin{align*}
A_{1}= & \frac{\left(n-2\right)^{2}}{2}\int_{\mathbb{R}_{+}^{n}}L_{1}(y)ny_{i}y_{j}\left(g^{ij}(\delta y)-\delta^{ij}\right)dy+O(\delta^{n-2})\\
= & \frac{n\left(n-2\right)^{2}}{2}\delta^{4}\int_{\mathbb{R}_{+}^{n}}L_{1}(y)\left(\frac{1}{20}\bar{R}_{ikjl,mp}+\frac{1}{15}\bar{R}_{iksl}\bar{R}_{jmsp}\right)y_{i}y_{j}y_{k}y_{l}y_{m}y_{p}dy\\
 & +\frac{n\left(n-2\right)^{2}}{2}\delta^{4}\int_{\mathbb{R}_{+}^{n}}L_{1}(y)\left(\frac{1}{2}R_{ninj,kl}+\frac{1}{3}\text{Sym}_{ij}(\bar{R}_{iksl}R_{nsnj})\right)y_{i}y_{j}y_{k}y_{l}y_{n}^{2}dy\\
 & +\frac{n\left(n-2\right)^{2}}{2}\delta^{4}\int_{\mathbb{R}_{+}^{n}}L_{1}(y)\frac{1}{12}\left(R_{ninj,nn}+8R_{nins}R_{nsnj}\right)y_{i}y_{j}y_{n}^{4}dy+O(\delta^{5}).
\end{align*}
By the symmetries of the curvature tensor (see \cite[Proof of Lemma 8, pages 15-16]{GMP}),
we have that 
\[
\int_{\mathbb{R}_{+}^{n}}L_{1}(y)\left(\frac{1}{20}\bar{R}_{ikjl,mp}+\frac{1}{15}\bar{R}_{iksl}\bar{R}_{jmsp}\right)y_{i}y_{j}y_{k}y_{l}y_{m}y_{p}dy=0
\]
and
\[
\delta^{4}\int_{\mathbb{R}_{+}^{n}}L_{1}(y)\frac{1}{3}\text{Sym}_{ij}(\bar{R}_{iksl}R_{nsnj})y_{i}y_{j}y_{k}y_{l}y_{n}^{2}dy=0,
\]
so 
\begin{align*}
A_{1}= & \frac{n\left(n-2\right)^{2}}{4}\delta^{4}\int_{\mathbb{R}_{+}^{n}}L_{1}(y)R_{ninj,kl}y_{i}y_{j}y_{k}y_{l}y_{n}^{2}dy\\
 & +\frac{n\left(n-2\right)^{2}}{24}\delta^{4}\int_{\mathbb{R}_{+}^{n}}L_{1}(y)\left(R_{ninj,nn}+8R_{nins}R_{nsnj}\right)y_{i}y_{j}y_{n}^{4}dy+O(\delta^{5}).
\end{align*}
We point out that, in the above integral only terms involving even
powers of $y_{s}$ survive. Moreover, by direct computation we have
that 
\[
3\int_{\mathbb{R}_{+}^{n}}L_{1}(y)y_{1}^{2}y_{2}^{2}y_{n}^{2}dy=\int_{\mathbb{R}_{+}^{n}}L_{1}(y)y_{1}^{4}y_{n}^{2}dy=\frac{3}{n^{2}-1}\int_{\mathbb{R}_{+}^{n}}L_{1}(y)|\bar{y}|^{4}y_{n}^{2}dy.
\]
So, for the first term we have 
\begin{multline*}
\int_{\mathbb{R}_{+}^{n}}L_{1}(y)R_{ninj,kl}y_{i}y_{j}y_{k}y_{l}y_{n}^{2}dy=\sum_{i}R_{nini,ii}\int_{\mathbb{R}_{+}^{n}}L_{1}(y)y_{1}^{4}y_{n}^{2}dy\\
+\left(\sum_{i\neq k}R_{nini,kk}+\sum_{i\neq j}R_{ninj,ij}+\sum_{i\neq j}R_{ninj,ji}\right)\int_{\mathbb{R}_{+}^{n}}L_{1}(y)y_{1}^{2}y_{2}^{2}y_{n}^{2}dy\\
=\left(3\sum_{i}R_{nini,ii}+\sum_{i\neq k}R_{nini,kk}+\sum_{i\neq j}R_{ninj,ij}+\sum_{i\neq j}R_{ninj,ji}\right)\int_{\mathbb{R}_{+}^{n}}L_{1}(y)y_{1}^{4}y_{n}^{2}dy\\
=\left(\sum_{i,k}R_{nini,kk}+\sum_{i,j}R_{ninj,ij}+\sum_{i,j}R_{ninj,ji}\right)\frac{1}{n^{2}-1}\int_{\mathbb{R}_{+}^{n}}L_{1}(y)|\bar{y}|^{4}y_{n}^{2}dy
\end{multline*}
By (\ref{eq:Ricci}), $R_{nn,kk}=0$ for all $k=1,\dots,n-1$, and
since the curvature tensor is at least $C^{2}$, we have finally
\[
\delta^{4}\frac{n\left(n-2\right)^{2}}{4}\int_{\mathbb{R}_{+}^{n}}L_{1}(y)R_{ninj,kl}y_{i}y_{j}y_{k}y_{l}y_{n}^{2}dy=\delta^{4}\frac{n\left(n-2\right)^{2}}{2(n^{2}-1)}R_{ninj,ji}\int_{\mathbb{R}_{+}^{n}}L_{1}(y)|\bar{y}|^{4}y_{n}^{2}dy.
\]
On the other hand, by (\ref{eq:Rnnnn}) we have
\begin{multline*}
\int_{\mathbb{R}_{+}^{n}}L_{1}(y)\left(R_{ninj,nn}+8R_{nins}R_{nsnj}\right)y_{i}y_{j}y_{n}^{4}dy\\
=\left(R_{nn,nn}+8R_{nins}R_{nsni}\right)\int_{\mathbb{R}_{+}^{n}}L_{1}(y)y_{1}^{2}y_{n}^{4}dy=6R_{nins}^{2}\int_{\mathbb{R}_{+}^{n}}L_{1}(y)y_{1}^{2}y_{n}^{4}dy\\
\frac{6}{n-1}R_{nins}^{2}\int_{\mathbb{R}_{+}^{n}}L_{1}(y)|\bar{y}|^{2}y_{n}^{4}dy
\end{multline*}
so, finally
\begin{align}
A_{1}= & \delta^{4}\frac{n\left(n-2\right)^{2}}{2(n^{2}-1)}R_{ninj,ji}\int_{\mathbb{R}_{+}^{n}}L_{1}(y)|\bar{y}|^{4}y_{n}^{2}dy\label{eq:A1}\\
 & +\delta^{4}\frac{n\left(n-2\right)^{2}}{4(n-1)}R_{nins}^{2}\int_{\mathbb{R}_{+}^{n}}L_{1}(y)|\bar{y}|^{2}y_{n}^{4}dy+O(\delta^{5}).\nonumber 
\end{align}
Similarly, for $A_{2}$ there are only the fourth order terms surviving,
and again we proceed by symmetry, using again (\ref{eq:Ricci}) and
(\ref{eq:Rnnnn}), obtaining

\begin{align*}
A_{2}= & -\delta^{4}\frac{\left(n-2\right)^{2}}{2}\int_{\mathbb{R}_{+}^{n}}L_{2}(y)\left(\frac{1}{20}\bar{R}_{ikil,mp}+\frac{1}{15}\bar{R}_{iksl}\bar{R}_{imsp}\right)y_{k}y_{l}y_{m}y_{p}dy\\
 & -\delta^{4}\frac{\left(n-2\right)^{2}}{2}\int_{\mathbb{R}_{+}^{n}}L_{2}(y)\left(\frac{1}{2}R_{nini,kl}+\frac{1}{3}\text{Sym}_{ii}(\bar{R}_{iksl}R_{nsni})\right)y_{n}^{2}y_{k}y_{l}dy\\
 & -\delta^{4}\frac{\left(n-2\right)^{2}}{2}\int_{\mathbb{R}_{+}^{n}}L_{2}(y)\left[\frac{1}{3}R_{nini,nk}y_{n}^{3}y_{k}+\frac{1}{12}\left(R_{nini,nn}+8R_{nins}R_{nsni}\right)y_{n}^{4}\right]dy+O(\delta^{5})\\
= & -\delta^{4}\frac{\left(n-2\right)^{2}}{2}\int_{\mathbb{R}_{+}^{n}}L_{2}(y)\left(\frac{1}{20}\bar{R}_{kl,mp}+\frac{1}{15}\bar{R}_{iksl}\bar{R}_{imsp}\right)y_{k}y_{l}y_{m}y_{p}dy\\
 & -\delta^{4}\frac{\left(n-2\right)^{2}}{2}\int_{\mathbb{R}_{+}^{n}}L_{2}(y)\frac{1}{12}\left(R_{nn,nn}+8R_{nins}^{2}\right)y_{n}^{4}dy+O(\delta^{5})\\
= & -\delta^{4}\left(\frac{\left(n-2\right)^{2}}{40}\bar{R}_{kl,mp}+\frac{\left(n-2\right)^{2}}{30}\bar{R}_{iksl}\bar{R}_{imsp}\right)\int_{\mathbb{R}_{+}^{n}}L_{2}(y)y_{k}y_{l}y_{m}y_{p}dy\\
 & -\delta^{4}\frac{\left(n-2\right)^{2}}{4}R_{nins}^{2}\int_{\mathbb{R}_{+}^{n}}L_{2}(y)y_{n}^{4}dy+O(\delta^{5}).
\end{align*}
and, similarly,
\begin{align*}
A_{3}= & -\delta^{4}\frac{\left(n-2\right)^{2}}{2}\int_{\mathbb{R}_{+}^{n}}L_{2}(y)\left(\frac{1}{20}\bar{R}_{ikjl,mp}+\frac{1}{15}\bar{R}_{iksl}\bar{R}_{jmsp}\right)\partial_{i}\left(y_{k}y_{l}y_{m}y_{p}\right)y_{j}dy\\
 & -\delta^{4}\frac{\left(n-2\right)^{2}}{2}\int_{\mathbb{R}_{+}^{n}}L_{2}(y)\left(\frac{1}{2}R_{ninj,kl}+\frac{1}{3}\text{Sym}_{ij}(\bar{R}_{iksl}R_{nsnj})\right)y_{n}^{2}\partial_{i}\left(y_{k}y_{l}\right)y_{j}dy\\
 & -\delta^{4}\frac{\left(n-2\right)^{2}}{2}\int_{\mathbb{R}_{+}^{n}}L_{2}(y)\frac{1}{3}R_{ninj,nk}y_{n}^{3}\partial_{i}(y_{k})y_{j}dy+O(\delta^{5})\\
= & \delta^{4}\frac{\left(n-2\right)^{2}}{2}\int_{\mathbb{R}_{+}^{n}}L_{2}(y)\left(\frac{1}{20}\bar{R}_{ikjl,mp}+\frac{1}{15}\bar{R}_{iksl}\bar{R}_{jmsp}\right)\partial_{i}\left(y_{k}y_{l}y_{m}y_{p}\right)y_{j}dy\\
 & -\delta^{4}\frac{\left(n-2\right)^{2}}{2}\int_{\mathbb{R}_{+}^{n}}L_{2}(y)\left(\frac{1}{2}R_{ninj,kl}\right)y_{n}^{2}\partial_{i}\left(y_{k}y_{l}\right)y_{j}dy+O(\delta^{5})\\
= & +\delta^{4}\left(\frac{\left(n-2\right)^{2}}{40}\bar{R}_{jl,mp}+\frac{\left(n-2\right)^{2}}{30}\bar{R}_{imsl}\bar{R}_{ijsp}\right)\int_{\mathbb{R}_{+}^{n}}L_{2}(y)y_{l}y_{m}y_{p}y_{j}dy\\
 & -\delta^{4}\frac{\left(n-2\right)^{2}}{2}R_{ninj,ij}\int_{\mathbb{R}_{+}^{n}}L_{2}(y)y_{n}^{2}y_{j}^{2}dy+O(\delta^{5})
\end{align*}
and, up to relabelling, we have 
\begin{align}
A_{2}+A_{3} & =-\delta^{4}\frac{\left(n-2\right)^{2}}{4}R_{nins}^{2}\int_{\mathbb{R}_{+}^{n}}L_{2}(y)y_{n}^{4}dy\label{eq:A2+A3}\\
 & -\delta^{4}\frac{\left(n-2\right)^{2}}{2}R_{ninj,ij}\int_{\mathbb{R}_{+}^{n}}L_{2}(y)y_{n}^{2}y_{j}^{2}dy+O(\delta^{5})\nonumber \\
= & -\delta^{4}\frac{\left(n-2\right)^{2}}{4}R_{nins}^{2}\int_{\mathbb{R}_{+}^{n}}L_{2}(y)y_{n}^{4}dy\nonumber \\
 & -\delta^{4}\frac{\left(n-2\right)^{2}}{2(n-1)}R_{ninj,ij}\int_{\mathbb{R}_{+}^{n}}L_{2}(y)|\bar{y}|^{2}y_{n}^{2}dy+O(\delta^{5}).\nonumber 
\end{align}
 Finally, by (\ref{eq:Rii}) we have
\begin{align*}
A_{4} & =\delta^{4}\frac{\left(n-2\right)^{2}}{96(n-1)^{2}}|\bar{W}(q)|^{2}\int_{\mathbb{R}_{+}^{n}}L_{3}(y)|\bar{y}|^{2}dy\\
 & -\delta^{4}\frac{\left(n-2\right)^{2}}{16(n-1)}\partial_{tt}^{2}\bar{R}_{g_{q}}(q)\int_{\mathbb{R}_{+}^{n}}L_{3}(y)y_{n}^{2}dy
\end{align*}
and by (\ref{eq:Rtt}) we conclude
\begin{align}
A_{4} & =\delta^{4}\frac{\left(n-2\right)^{2}}{96(n-1)^{2}}|\bar{W}(q)|^{2}\int_{\mathbb{R}_{+}^{n}}L_{3}(y)|\bar{y}|^{2}dy\label{eq:A4}\\
 & +\delta^{4}\frac{\left(n-2\right)^{2}}{8(n-1)}R_{ninj}^{2}\int_{\mathbb{R}_{+}^{n}}L_{3}(y)y_{n}^{2}dy\nonumber \\
 & +\delta^{4}\frac{\left(n-2\right)^{2}}{8(n-1)}R_{ninj,ij}\int_{\mathbb{R}_{+}^{n}}L_{3}(y)y_{n}^{2}dy.\nonumber 
\end{align}
We want now collect the similar terms, using the result of Lemma \ref{lem:integraligamma}
to estimate all integrals. All terms containing $R_{ninj,ij}$ in
(\ref{eq:A1}), (\ref{eq:A2+A3}) and (\ref{eq:A4}) add up to
\begin{multline}
\delta^{4}R_{ninj,ji}\frac{n-2}{2(n-1)}\\
\times\left[\frac{n\left(n-2\right)}{(n+1)}\int_{\mathbb{R}_{+}^{n}}L_{1}(y)|\bar{y}|^{4}y_{n}^{2}dy-\left(n-2\right)\int_{\mathbb{R}_{+}^{n}}L_{2}(y)|\bar{y}|^{2}y_{n}^{2}dy+\frac{\left(n-2\right)}{4}\int_{\mathbb{R}_{+}^{n}}L_{3}(y)y_{n}^{2}dy\right]\\
=\delta^{4}R_{ninj,ji}\frac{(n-2)\omega_{n-2}I_{n}^{n}}{2(n-1)}\\
\times\left[\frac{12\left(n-2\right)}{(n-3)(n-4)(n-5)(n-6)}-\frac{20\left(n-2\right)}{(n-3)(n-4)(n-5)(n-6)}+\frac{8(n-2)}{(n-3)(n-4)(n-5)(n-6)}\right]\\
=0.\label{eq:Rninj,ij}
\end{multline}
Concerning terms containing $R_{ninj}^{2}$ we have 
\begin{multline}
\delta^{4}R_{ninj}^{2}\frac{\left(n-2\right)^{2}}{4}\\
\times\left[\frac{n}{(n-1)}\int_{\mathbb{R}_{+}^{n}}L_{1}(y)|\bar{y}|^{2}y_{n}^{4}dy-\int_{\mathbb{R}_{+}^{n}}L_{2}(y)y_{n}^{4}dy+\frac{1}{2(n-1)}\int_{\mathbb{R}_{+}^{n}}L_{3}(y)y_{n}^{2}dy\right]\\
\delta^{4}R_{ninj}^{2}\frac{\left(n-2\right)^{2}\omega_{n-2}I_{n}^{n}}{4}\\
\times\frac{1}{(n-1)(n-2)(n-3)(n-4)(n-5)(n-6)}\left[144-240+16(n-2)\right]\\
=\delta^{4}R_{ninj}^{2}\frac{4(n-2)(n-8)\omega_{n-2}I_{n}^{n}}{(n-1)(n-3)(n-4)(n-5)(n-6)}.\label{eq:R2ninj}
\end{multline}
In light of (\ref{eq:Rninj,ij}) and (\ref{eq:R2ninj}), we can conclude,
by (\ref{eq:A1}), (\ref{eq:A2+A3}) and (\ref{eq:A4}), 
\begin{align*}
R(U,U)= & \delta^{4}\frac{\left(n-2\right)^{2}}{96(n-1)^{2}}|\bar{W}(q)|^{2}\int_{\mathbb{R}_{+}^{n}}L_{3}(y)|\bar{y}|^{2}dy\\
 & +\delta^{4}R_{ninj}^{2}\frac{4(n-2)(n-8)\omega_{n-2}I_{n}^{n}}{(n-1)(n-3)(n-4)(n-5)(n-6)}\\
= & \delta^{4}|\bar{W}(q)|^{2}\frac{\left(n-2\right)^{2}\omega_{n-2}I_{n}^{n}}{6(n-1)(n-3)(n-5)(n-6)}\\
 & +\delta^{4}R_{ninj}^{2}\frac{4(n-2)(n-8)\omega_{n-2}I_{n}^{n}}{(n-1)(n-3)(n-4)(n-5)(n-6)}
\end{align*}
which ends the proof.
\end{proof}
\begin{lem}
\label{lem:R(U,vq)}For $n\ge8$ we have
\[
R(U,\delta^{2}\gamma_{q})+R(\delta^{2}\gamma_{q},U)=-2\delta^{4}\int_{\mathbb{R}_{+}^{n}}\gamma_{q}\Delta\gamma_{q}dy+o(\delta^{4}).
\]
\end{lem}
\begin{proof}
In light of (\ref{eq:gij}) and (\ref{eq:gradvq}), we have that 
\begin{align*}
R(U,\delta^{2}\gamma_{q})+R(\delta^{2}\gamma_{q},U)= & -\delta^{4}\int_{\mathbb{R}_{+}^{n}}\left(y_{b}\partial_{b}U+\frac{n-2}{2}U\right)\left[\frac{1}{3}\bar{R}_{ikjl}y_{k}y_{l}+R_{ninj}y_{n}^{2}\right]\partial_{i}\partial_{j}\gamma_{q}dy\\
 & -\delta^{4}\int_{\mathbb{R}_{+}^{n}}\left(y_{b}\partial_{b}\gamma_{q}+\frac{n-2}{2}\gamma_{q}\right)\left[\frac{1}{3}\bar{R}_{ikjl}y_{k}y_{l}+R_{ninj}y_{n}^{2}\right]\partial_{i}\partial_{j}Udy+O(\delta^{n-2})\\
= & -\delta^{4}\int_{\mathbb{R}_{+}^{n}}\left(y_{b}\partial_{b}U+\frac{n-2}{2}U\right)\left[\frac{1}{3}\bar{R}_{ikjl}y_{k}y_{l}+R_{ninj}y_{n}^{2}\right]\partial_{i}\partial_{j}\gamma_{q}dy\\
 & -\delta^{4}\int_{\mathbb{R}_{+}^{n}}y_{b}\partial_{b}\gamma_{q}\left[\frac{1}{3}\bar{R}_{ikjl}y_{k}y_{l}+R_{ninj}y_{n}^{2}\right]\partial_{i}\partial_{j}Udy\\
 & -\delta^{4}\int_{\mathbb{R}_{+}^{n}}\frac{n-2}{2}\gamma_{q}\left[\frac{1}{3}\bar{R}_{ikjl}y_{k}y_{l}+R_{ninj}y_{n}^{2}\right]\partial_{i}\partial_{j}Udy+o(\delta^{4})\\
=: & \delta^{4}(A_{1}+A_{2}+A_{3})+o(\delta^{4}).
\end{align*}
Immediately we have, by the choice of $\gamma_{q}$ (see (\ref{eq:vqdef})),
that 
\begin{equation}
A_{3}=\frac{n-2}{2}\int_{\mathbb{R}_{+}^{n}}\gamma_{q}\Delta\gamma_{q}.\label{eq:A3-1}
\end{equation}
We notice that, given any two functions $f,g$, we have, by (\ref{eq:Ricci})
and by the symmetries of the curvature tensor, that 
\[
\int_{\mathbb{R}_{+}^{n}}f\partial_{i}\left[\frac{1}{3}\bar{R}_{ikjl}y_{k}y_{l}+R_{ninj}y_{n}^{2}\right]\partial_{j}gdy=0.
\]
So, integrating by parts we have 
\begin{align*}
A_{2}= & \int_{\mathbb{R}_{+}^{n}}n\gamma_{q}\left[\frac{1}{3}\bar{R}_{ikjl}y_{k}y_{l}+R_{ninj}y_{n}^{2}\right]\partial_{i}\partial_{j}Udy\\
 & +\int_{\mathbb{R}_{+}^{n}}y_{b}\gamma_{q}\partial_{b}\left[\frac{1}{3}\bar{R}_{ikjl}y_{k}y_{l}+R_{ninj}y_{n}^{2}\right]\partial_{i}\partial_{j}Udy\\
 & +\int_{\mathbb{R}_{+}^{n}}y_{b}\gamma_{q}\left[\frac{1}{3}\bar{R}_{ikjl}y_{k}y_{l}+R_{ninj}y_{n}^{2}\right]\partial_{b}\partial_{i}\partial_{j}Udy\\
 & +\int_{\partial\mathbb{R}_{+}^{n}}y_{b}\nu_{b}\gamma_{q}\left[\frac{1}{3}\bar{R}_{ikjl}y_{k}y_{l}+R_{ninj}y_{n}^{2}\right]\partial_{i}\partial_{j}Udy.
\end{align*}
We notice that $y_{b}\nu_{b}=0$ on $\partial\mathbb{R}_{+}^{n}$.
Moreover, up to relabelling,
\begin{multline*}
\int_{\mathbb{R}_{+}^{n}}y_{b}\gamma_{q}\partial_{b}\left[\frac{1}{3}\bar{R}_{ikjl}y_{k}y_{l}+R_{ninj}y_{n}^{2}\right]\partial_{i}\partial_{j}Udy\\
=\int_{\mathbb{R}_{+}^{n}}y_{s}\gamma_{q}\partial_{s}\left[\frac{1}{3}\bar{R}_{ikjl}y_{k}y_{l}\right]\partial_{i}\partial_{j}Udy+\int_{\mathbb{R}_{+}^{n}}y_{n}\gamma_{q}\partial_{n}\left[R_{ninj}y_{n}^{2}\right]\partial_{i}\partial_{j}Udy\\
=2\int_{\mathbb{R}_{+}^{n}}\gamma_{q}\left[\frac{1}{3}\bar{R}_{ikjl}y_{k}y_{l}+R_{ninj}y_{n}^{2}\right]\partial_{i}\partial_{j}Udy=-2\int_{\mathbb{R}_{+}^{n}}\gamma_{q}\Delta\gamma_{q},
\end{multline*}
so 
\begin{equation}
A_{2}=-\left(n+2\right)\int_{\mathbb{R}_{+}^{n}}\gamma_{q}\Delta\gamma_{q}+\int_{\mathbb{R}_{+}^{n}}y_{b}\gamma_{q}\left[\frac{1}{3}\bar{R}_{ikjl}y_{k}y_{l}+R_{ninj}y_{n}^{2}\right]\partial_{b}\partial_{i}\partial_{j}Udy.\label{eq:A2-1}
\end{equation}
Finally, integrating by parts twice, we have, arguing as before,
\begin{align*}
A_{1}= & \int_{\mathbb{R}_{+}^{n}}\left(\partial_{i}\left(y_{b}\right)\partial_{b}U+y_{b}\partial_{i}\partial_{b}U+\frac{n-2}{2}\partial_{i}U\right)\left[\frac{1}{3}\bar{R}_{ikjl}y_{k}y_{l}+R_{ninj}y_{n}^{2}\right]\partial_{j}\gamma_{q}dy\\
 & +\int_{\mathbb{R}_{+}^{n}}\left(y_{b}\partial_{b}U+\frac{n-2}{2}U\right)\partial_{i}\left[\frac{1}{3}\bar{R}_{ikjl}y_{k}y_{l}+R_{ninj}y_{n}^{2}\right]\partial_{j}\gamma_{q}dy\\
 & +\int_{\partial\mathbb{R}_{+}^{n}}\nu_{i}\left(y_{b}\partial_{b}U+\frac{n-2}{2}U\right)\left[\frac{1}{3}\bar{R}_{ikjl}y_{k}y_{l}+R_{ninj}y_{n}^{2}\right]\partial_{j}\gamma_{q}dy\\
= & \int_{\mathbb{R}_{+}^{n}}\left(\frac{n}{2}\partial_{i}U+y_{b}\partial_{i}\partial_{b}U\right)\left[\frac{1}{3}\bar{R}_{ikjl}y_{k}y_{l}+R_{ninj}y_{n}^{2}\right]\partial_{j}\gamma_{q}dy\\
= & -\int_{\mathbb{R}_{+}^{n}}\left(\frac{n}{2}\partial_{j}\partial_{i}U+\partial_{j}\left(y_{b}\right)\partial_{i}\partial_{b}U+y_{b}\partial_{j}\partial_{i}\partial_{b}U\right)\left[\frac{1}{3}\bar{R}_{ikjl}y_{k}y_{l}+R_{ninj}y_{n}^{2}\right]\gamma_{q}dy\\
 & -\int_{\mathbb{R}_{+}^{n}}\left(\frac{n}{2}\partial_{i}U+y_{b}\partial_{i}\partial_{b}U\right)\partial_{j}\left[\frac{1}{3}\bar{R}_{ikjl}y_{k}y_{l}+R_{ninj}y_{n}^{2}\right]\gamma_{q}dy\\
 & -\int_{\partial\mathbb{R}_{+}^{n}}\nu_{j}\left(\frac{n}{2}\partial_{i}U+y_{b}\partial_{i}\partial_{b}U\right)\left[\frac{1}{3}\bar{R}_{ikjl}y_{k}y_{l}+R_{ninj}y_{n}^{2}\right]\gamma_{q}dyx\\
= & -\int_{\mathbb{R}_{+}^{n}}\left(\left(\frac{n}{2}+1\right)\partial_{j}\partial_{i}U+y_{b}\partial_{j}\partial_{i}\partial_{b}U\right)\left[\frac{1}{3}\bar{R}_{ikjl}y_{k}y_{l}+R_{ninj}y_{n}^{2}\right]\gamma_{q}dy\\
= & \left(\frac{n}{2}+1\right)\int_{\mathbb{R}_{+}^{n}}\gamma_{q}\Delta\gamma_{q}dy-\int_{\mathbb{R}_{+}^{n}}y_{b}\partial_{j}\partial_{i}\partial_{b}U\left[\frac{1}{3}\bar{R}_{ikjl}y_{k}y_{l}+R_{ninj}y_{n}^{2}\right]\gamma_{q}dy.
\end{align*}
Recalling (\ref{eq:A3-1}) and (\ref{eq:A2-1}) we conclude 

\begin{align*}
A_{1}+A_{2}+A_{3} & =-2\int_{\mathbb{R}_{+}^{n}}\gamma_{q}\Delta\gamma_{q}
\end{align*}
which gives the proof.
\end{proof}
\begin{lem}
\label{lem:Raltri}For $n\ge8$ we have
\begin{align*}
R(\delta^{2}\gamma_{q},\delta^{2}\gamma_{q}) & =O(\delta^{6})\\
R(w_{i},w_{i}) & =O(\delta^{6})\\
R(U,w_{i})+R(w_{i},U) & =O(\delta^{5})\\
R(\delta^{2}\gamma_{q},w_{i})+R(w_{i},\delta^{2}\gamma_{q}) & =O(\delta^{5})
\end{align*}
\end{lem}
\begin{proof}
By direct computation, using the decay of the standard bubble $U$,
estimate (\ref{eq:gradvq}), (\ref{eq:gij}) and Proposition \ref{prop:stimawi}.
\end{proof}
Here we focus on the Weyl tensor of $M$, proving a result which is
in the spirit of Weyl vanishing conjecture.
\begin{prop}
\label{prop:7.1}Let $x_{i}\rightarrow x_{0}$ be an isolated simple
blow up point for $u_{i}$ solutions of (\ref{eq:Prob-i}). Then
\begin{enumerate}
\item If $n=8$ then $|\bar{W}(x_{0})|=0.$
\item If $n>8$ then $|W(x_{0})|=0.$
\end{enumerate}
\end{prop}
\begin{proof}
By Proposition \ref{prop:4.3} and Proposition \ref{prop:Lemma 4.4},
and since $M_{i}=\delta_{i}^{\frac{1}{1-p_{i}}}$, we have,
\begin{align*}
P(u_{i},r):= & \frac{1}{M_{i}^{2\lambda_{i}}}\int\limits _{\partial^{+}B_{r}^{+}}\left(\frac{n-2}{2}M_{i}^{\lambda_{i}}u_{i}\frac{\partial M_{i}^{\lambda_{i}}u_{i}}{\partial r}-\frac{r}{2}|\nabla M_{i}^{\lambda_{i}}u_{i}|^{2}+r\left|\frac{\partial M_{i}^{\lambda_{i}}u_{i}}{\partial r}\right|^{2}\right)d\sigma_{r}\\
 & +\frac{r(n-2)}{\left(p_{i}+1\right)M_{i}^{\lambda_{i}(p_{i}+1)}}\int\limits _{\partial(\partial'B_{r}^{+})}f^{-\tau}\left(M_{i}^{\lambda_{i}}u_{i}\right)^{p_{i}+1}d\bar{\sigma}_{g}\\
\le & \frac{C}{M_{i}^{2\lambda_{i}}}\le C\delta_{i}^{\frac{2}{p_{i}-1}\lambda_{i}}\le C\delta_{i}^{n-2}.
\end{align*}
On the other hand, recalling Proposition \ref{prop:segno}, we have 

\[
P(u_{i},r)\ge\delta_{i}^{4}\frac{(n-2)\omega_{n-2}I_{n}^{n}}{(n-1)(n-3)(n-5)(n-6)}\left[\frac{\left(n-2\right)}{6}|\bar{W}(x_{i})|^{2}+\frac{4(n-8)}{(n-4)}R_{nlnj}^{2}(x_{i})\right]+o(\delta_{i}^{4}),
\]
so we get $|\bar{W}(x_{i})|\le\delta_{i}^{2}$ if $n=8$, and $\left[\frac{\left(n-2\right)}{6}|\bar{W}(x_{i})|^{2}+\frac{4(n-8)}{(n-4)}R_{nlnj}^{2}(x_{i})\right]\le\delta_{i}^{2}$
if $n>8$. For the case $n>8$ we recall that when the boundary is
umbilic $W(q)=0$ if and only if $\bar{W}(q)=0$ and $R_{nlnj}(q)=0$
(see \cite[page 1618]{M1}), and we conclude the proof. 
\end{proof}
\begin{rem}
\label{rem:P'}Let $x_{i}\rightarrow x_{0}$ be an isolated blow up
point for $u_{i}$ solutions of (\ref{eq:Prob-i}). We set 
\begin{equation}
P'\left(u,r\right):=\int\limits _{\partial^{+}B_{r}^{+}}\left(\frac{n-2}{2}u\frac{\partial u}{\partial r}-\frac{r}{2}|\nabla u|^{2}+r\left|\frac{\partial u}{\partial r}\right|^{2}\right)d\sigma_{r},\label{eq:P'def}
\end{equation}
 so
\[
P(u_{i},r)=P'(u_{i},r)+\frac{r(n-2)}{p_{i}+1}\int\limits _{\partial(\partial'B_{r}^{+})}f_{i}^{-\tau_{i}}u_{i}^{p_{i}+1}d\bar{\sigma}_{g}
\]
and, keeping in mind that for $i$ large $M_{i}u_{i}\le C|y|^{2-n}$
by Proposition \ref{prop:4.3}, and since $f_{i}^{-\tau_{i}}\rightarrow0$
and $p_{i}\rightarrow\frac{n}{n-2}$, we have 
\begin{align}
\left|r\int\limits _{\partial(\partial'B_{r}^{+})}f_{i}^{-\tau_{i}}u_{i}^{p_{i}+1}d\bar{\sigma}_{g}\right| & \le\frac{Cr}{M_{i}^{p_{i}+1}}\int_{\begin{array}{c}
y_{n}=0\\
|\bar{y}|=r
\end{array}}|y|^{(p_{i}+1)(2-n)}d\bar{\sigma}_{g}\nonumber \\
 & \le\frac{Cr^{(p_{i}+1)(2-n)+1}}{M_{i}^{p_{i}+1}}\int_{\begin{array}{c}
y_{n}=0\\
|\bar{y}|=r
\end{array}}1d\bar{\sigma}_{g}\le C(r)\delta_{i}^{n-2}\label{eq:5.15-1}
\end{align}
for $i$ sufficiently large.

Using Proposition \ref{prop:segno}, (\ref{eq:5.15-1}), and since
$n\ge8$ we get 
\begin{equation}
P'(u_{i},r)=P(u_{i},r)-\frac{r(n-2)}{p_{i}+1}\int\limits _{\partial(\partial'B_{r}^{+})}f_{i}^{-\tau_{i}}u_{i}^{p_{i}+1}d\bar{\sigma}_{g}\ge A\delta_{i}^{4}+o(\delta^{4})\label{eq:stimaP'}
\end{equation}
where $A>0$.
\end{rem}
\begin{prop}
\label{prop:isolato->semplice}Let $x_{i}\rightarrow x_{0}$ be an
isolated blow up point for $u_{i}$ solutions of (\ref{eq:Prob-i}).
Assume $n=8$ and $|\bar{W}(x_{0})|\neq0$ or $n>8$ and $|W(x_{0})|\neq0$.
Then $x_{0}$ is isolated simple. 
\end{prop}
\begin{proof}
Set $w_{i}(y)=u_{i}(\psi_{i}(y))$ where $\psi_{i}$ are, as usual,
the Fermi coordinates at $x_{i}$ defined in $B_{\rho}(0)$. By assumption
$0$ is an isolated blow up point for $w_{i}$. By contradiction,
suppose that $0$ is not isolated simple. Take $R_{i}\rightarrow\infty$
and define $r_{i}:=R_{i}w_{i}^{1-p_{i}}(0)$. Then the function $r\rightarrow r^{\frac{1}{p_{i}-1}}\bar{w}_{i}(r)$
has exactly one critical point in $(0,r_{i})$. By Definition \ref{def:isolatedsimple},
since $x_{0}$ is not an isolated simple blow up point, there exist
at least two critical points of the function $r\mapsto r^{\frac{1}{p_{i}-1}}\bar{w}_{i}$
in an interval $(0,\bar{\rho}_{i})$ with $\bar{\rho}_{i}\rightarrow0$.
So, if $\rho{}_{i}$ is the second critical point, we have $0<r_{i}\le\rho_{i}<\bar{\rho}_{i}$.
We set 
\begin{equation}
v_{i}(y)=\rho_{i}^{\frac{1}{p_{i}-1}}w_{i}(\rho_{i}y)\text{ for }y\in B_{\rho/\bar{\rho}_{i}}^{+}(0).\label{eq:v-isolated-simple}
\end{equation}
 By construction we have that $0$ is an isolated simple blow up point
for $v_{i}$. Indeed, by definition of $r_{i}$,
\begin{equation}
v_{i}(0)=\rho_{i}^{\frac{1}{p_{i}-1}}w_{i}(0)=\left(\frac{\rho_{i}}{r_{i}}R_{i}\right)^{\frac{1}{p_i-1}}\ge R_{i}^{\frac{1}{p_i-1}}\rightarrow+\infty.\label{eq:viblowup}
\end{equation}
Moreover, the function $r\mapsto r^{\frac{1}{p_{i}-1}}\bar{v}_{i}(r)=\left(\rho_{i}r\right)^{\frac{1}{p_{i}-1}}\bar{w}_{i}(\rho_{i}r)$
has exactly one critical point in $(0,1)$.

By the first claim of Proposition \ref{prop:4.3} we have that $v_{i}(0)v_{i}(x)$
is uniformly bounded in the compact sets of $\mathbb{R}_{+}^{n}\smallsetminus\left\{ 0\right\} $.
Taking in account that $u_{i}$ solves (\ref{eq:Prob-i}) and $v_{i}$
solves (\ref{eq:Prob-hat}), we can prove that $v_{i}(0)v_{i}(x)\rightarrow G$
in $C_{\text{loc}}^{2}(\mathbb{R}_{+}^{n}\smallsetminus\left\{ 0\right\} )$,
where $G$ safisfies 
\[
\left\{ \begin{array}{ccc}
\Delta G=0 &  & \text{ in }\mathbb{R}_{+}^{n}\smallsetminus\left\{ 0\right\} \\
\partial_{n}G=0 &  & \text{ on }\partial\mathbb{R}_{+}^{n}\smallsetminus\left\{ 0\right\} 
\end{array}\right..
\]
It is well known that $G=a|y|^{2-n}+b(y)$, with $b$ harmonic on
$\mathbb{R}_{+}^{n}$ with Neumann boundary condition. Moreover, by
the second claim of Proposition \ref{prop:4.3}, we can show that
$a>0$. Since $G>0$, the function $b$ is non negative at infinity,
and by Liouville theorem this implies that $b$ is a constant function.
Moreover, by the equality $\left.\frac{d}{dr}r^{\frac{1}{p_{i}-1}}\bar{v}_{i}(r)\right|_{r=1}=0$,
we have $\left.\frac{d}{dr}\left(r\right)^{\frac{1}{p_{i}-1}}\bar{G}(r)\right|_{r=1}=0$,
that implies $a=b>0$. 

At this point, defined $P'\left(u,r\right)$ as in (\ref{eq:P'def})
and proceeding as in Remark \ref{rem:P'}, in analogy with (\ref{eq:stimaP'})
we have
\begin{align}
P'(v_{i}(0)v_{i},r)\ge P(v_{i}(0)v_{i},r) & -v_{i}(0)^{2}O(\delta_{i}^{n-2})\ge v_{i}(0)^{2}\left[A\delta_{i}^{4}+o(\delta_{i}^{4})\right]>0\label{eq:Pmaggiore}
\end{align}
for $i$ sufficiently large. 

On the other hand a direct computation shows that 
\begin{equation}
\lim_{i\rightarrow\infty}P'(v_{i}(0)v_{i},r)=P'(G,r)<0\label{eq:Ppos}
\end{equation}
provided $r$ sufficiently small, which contradicts (\ref{eq:Pmaggiore}).
\end{proof}

\section{\label{sec:splitting}A splitting lemma}

We start recalling a result which is analogous to \cite[Proposition 5.1]{LZ},
\cite[Lemma 3.1]{SZ}, \cite[Proposition 1.1]{HL} and \cite[Proposition 4.2]{Al},
which we refer for the proof. 
\begin{prop}
\label{prop:4.2}Given $\beta>0$ and $R>0$ there exist two constants
$C_{0},C_{1}>0$ (depending on $\beta$, $R$ and $(M,g)$) such that,
if $u$ is a solution of 
\begin{equation}
\left\{ \begin{array}{cc}
L_{g}u=0 & \text{ in }M\\
B_{g}u+(n-2)f^{-\tau}u^{p}=0 & \text{ on }\partial M
\end{array}\right.\label{eq:Prob-p}
\end{equation}
and $\max_{\partial M}u>C_{0}$, then $\tau:=\frac{n}{n-2}-p<\beta$
and there exist $q_{1},\dots,q_{N}\in\partial M$, with $N=N(u)\ge1$
with the following properties: for $j=1,\dots,N$ 
\begin{enumerate}
\item set $r_{j}:=Ru(q_{j})^{1-p}$, then $\left\{ B_{r_{j}}\cap\partial M\right\} _{j}$
are a disjoint collection;
\item we have $\left|u(q_{j})^{-1}u(\psi_{j}(y))-U(u(q_{j})^{p-1}y)\right|_{C^{2}(B_{2r_{j}}^{+})}<\beta$
(here $\psi_{j}$ are the Fermi coordinates at point $q_{j}$;
\item we have
\begin{align}
u(x)d_{\bar{g}}\left(x,\left\{ q_{1},\dots,q_{n}\right\} \right)^{\frac{1}{p-1}}\le C_{1} & \text{ for all }x\in\partial M\label{eq:Claim3-1}\\
u(q_{j})d_{\bar{g}}\left(q_{j},q_{k}\right)^{\frac{1}{p-1}}\ge C_{0} & \text{ for any }j\neq k.\label{eq:Claim3-2}
\end{align}
Here $\bar{g}$ is the geodesic distance on $\partial M$.
\end{enumerate}
\end{prop}
This proposition states that $u$ is well approximated in strong norms
by standard bubbles in disjoint balls $B_{r_{1}},\dots B_{r_{N}}$
centered on $\partial M$. It is not yet the compactness result we
need, since we have to consider, when passing to sequence of solutions,
interaction between bubbles. The next Proposition rules out possible
accumulation of bubbles, that implies that only isolated blow up points
may occur to a blowing up sequence of solution. 
\begin{prop}
\label{prop:splitting}Assume $n\ge8$. Given $\beta,R>0$, consider
$C_{0},C_{1}$ as in the previous proposition. Assume $W(x)\neq0$
for any $x\in\partial M$ if $n>8$, or $\bar{W}(x)\neq0$ for any
$x\in\partial M$ if $n=8$. Then there exists $d=d(\beta,R)$ such
that, for any $u$ solution of (\ref{eq:Prob-p}) with $\max_{\partial M}u>C_{0}$,
we have 
\[
\min_{\begin{array}{c}
i\neq j\\
1\le i,j\le N(u)
\end{array}}d_{\bar{g}}(q_{i}(u),q_{j}(u))\ge d,
\]
where $q_{1}(u),\dots q_{N}(u)$ and $N=N(u)$ are given in the previous
proposition. 
\end{prop}
\begin{proof}
We prove the result for $N(u)=2$. The general case follows easily. 

We argue by contradiction: we suppose that there exists a sequence
of solutions $\{u_{i}\}_{i}$ of problem (\ref{eq:Prob-2}) such that
(after relabelling the indices) we have two sequence of points $q_{1}^{i},q_{2}^{i}\in\partial M$
and a point $q_{0}\in\partial M$ with $q_{1}^{i},q_{2}^{i}\rightarrow q_{0}$.
Define 
\[
\sigma_{i}:=d_{\bar{g}}(q_{1}^{i},q_{2}^{i})=\min_{a\neq b}d_{\bar{g}}(q_{a}^{i},q_{b}^{i})
\]
Now we use Fermi coordinates $\psi_{i}:B_{\rho}^{+}\rightarrow M$
centered at $q_{1}^{i}$ and we set
\[
v_{i}(y)=\sigma_{i}^{\frac{1}{p_{i}-1}}u_{i}\left(\psi_{i}(\sigma_{i}y)\right),\ y\in B_{\sigma_{i}^{-1/2}}^{+}.
\]
For $k=1,2$ we define $y_{k}^{i}$ as the point in $B_{\sigma_{i}^{-1/2}}^{+}$
such that $\psi_{i}(\sigma_{i}y_{k}^{i})=q_{k}^{1}$. Of course we
have $y_{1}^{i}=0$.

By equation (\ref{eq:Claim3-1}) of Proposition \ref{prop:4.2}, and
by definition of $v_{i},q_{k}^{i}$, we have that 
\begin{equation}
v_{i}(y_{k}^{i})=d_{\bar{g}}(q_{1}^{i},q_{2}^{i})^{\frac{1}{p_{i}-1}}u_{i}(q_{k}^{i})\ge C_{0}\text{ for }k=1,2.\label{eq:*}
\end{equation}
 \emph{Step 1}. $v_{i}(y_{1}^{i}),v_{i}(y_{2}^{i})\rightarrow\infty$.

We proceed by contradiction. We first suppose that $v_{i}(y_{2}^{i})$
is bounded while $v_{i}(y_{1}^{i})\rightarrow\infty$. By equation
(\ref{eq:Claim3-2}) of Proposition \ref{prop:4.2} we have that $y_{1}^{i}=0$
is an isolated simple blow up point. Then, by Proposition \ref{prop:isolato->semplice},
is also isolated simple. Since $v_{i}(y_{2}^{i})$ is uniformly bounded
by assumption, by an Harnack type inequality (\cite[Prop. 9.3]{Al})
we have that $v_{i}$ is uniformly bounded in a neighborhood of $y_{2}^{i}$.
Then $y_{2}^{i}$ is a regular point, and since $y_{1}^{i}$ is an
isolated simple blow up point, by Proposition \ref{prop:4.3}, Claim
1, we obtain that $v_{i}(y_{2}^{i})\rightarrow0$. This contradicts
equation (\ref{eq:*}). If we switch the role of $y_{1}^{k}$ and
$y_{2}^{k}$ the contradiction follows analogously, so we have only
to rule out the case in which both $v_{i}(y_{1}^{i}),v_{i}(y_{2}^{i})$
remain bounded. In this case we can prove that $v_{i}$ converge in
$C_{\text{loc}}^{2}(\mathbb{R}_{+}^{n})$ to $v$, a solution of 
\[
\left\{ \begin{array}{cc}
\Delta v=0 & \text{ in }\mathbb{R}_{+}^{n}\\
\frac{\partial v}{\partial\nu}+(n-2)f^{p_{0}-\frac{n}{n-2}}u^{p_{0}}=0 & \text{ on }\partial\mathbb{R}_{+}^{n}
\end{array}\right.
\]
where $p_{0}=\lim_{i}p_{i}$. Then, by Liouville theorem, we get $v\equiv0$,
which again contradicts (\ref{eq:*}), and Step 1 is proved. 

\noindent \emph{Step 2}. Conclusion. 

By Step 1 we have that both $y_{1}^{i}$ and $y_{2}^{i}$ are isolated
blow up points, and thus isolated simple blow up points for $v_{i}$.
At this point we proceed as in Proposition \ref{prop:isolato->semplice}
and we have that $v_{i}(0)v_{i}(x)\rightarrow G(y)=a_{1}|y|^{2-n}+a_{2}|y-y_{2}|^{2-n}+b(y)$
in $C_{\text{loc}}^{2}(\mathbb{R}_{+}^{n}\smallsetminus\left\{ 0,y_{2}\right\} )$,
where $y_{2}=\lim_{i}y_{2}^{i}$ , $b(y)$ is an harmonic function
on $\mathbb{R}_{+}^{n}\smallsetminus\left\{ 0,y_{2}\right\} $ with
Neumann boundary condition and $a_{1},a_{2}>0$. 

By the maximum principle $b(y)\ge0$, so near $0$ we have 
\begin{equation}
v_{i}(0)v_{i}(x)=a_{1}|y|^{2-n}+b+O(|y|)\label{eq:splittingfinale}
\end{equation}
for some $b>0$. As in Proposition, \ref{prop:isolato->semplice}
equation (\ref{eq:splittingfinale}) contradicts the sign condition
given by the Pohozaev inequality, since we supposed $W(x)\neq0$ if
$n>8$ or $\bar{W}(x)\neq0$ if $n=8$. This concludes the proof.
\end{proof}
\begin{rem}\label{rem:Nbar}
Notice that, by the above proposition, there exists $\bar{N}$ such
that $N(u)\le\bar{N}<+\infty$ for all $u$.
\end{rem}

\section{\label{sec:Proof}Proof of the main result}
\begin{proof}[Proof of Theorem \ref{thm:main}]
By contradiction, suppose that $x_{i}\rightarrow x_{0}$ is a blow
up point for $u_{i}$ solutions of (\ref{eq:P-conf}). Let $q_{1}^{i},\dots q_{N(u_{i})}^{i}$
the sequence of points given by Proposition \ref{prop:4.2}, with $N(u_{i})\le\bar{N}$ by Remark \ref{rem:Nbar}. By Claim
3 of Proposition \ref{prop:4.2} there exists a sequence of indices
$k_{i}\in1,\dots N$ such that $d_{\bar{g}}\left(x_{i},q_{k_{i}}^{i}\right)\rightarrow0$.
Up to relabeling, we say $k_{i}=1$ for all $i$. Then also $q_{1}^{i}\rightarrow x_{0}$
is a blow up point for $u_{i}$. By Proposition \ref{prop:splitting} and
Proposition \ref{prop:4.2}  we have that $q_{1}^{i}\rightarrow x_{0}$ is an isolated simple blow
up point for $u_{i}$, and by Proposition \ref{prop:isolato->semplice}
we have that $q_{1}^{i}\rightarrow x_{0}$ is also isolated simple. 
Finally by Proposition \ref{prop:7.1} we deduce
that $\bar{W}(x_{0})=0$ if $n=8$ or that $W(x_{0})=0$ if $n>8$,
which contradicts the assumption of this theorem and proves the result. 
\end{proof}

\section{Appendix}
\begin{proof}[Proof of Lemma \ref{lem:vq} ]
We follow the strategy of \cite[Prop 5.1]{Al}. To prove the existence
of a solution of (\ref{eq:vqdef}) we have to show that the given
term $\left[\frac{1}{3}\bar{R}_{ijkl}(q)z_{k}z_{l}+R_{ninj}(q)t^{2}\right]\partial_{ij}^{2}U$
is $L^{2}$-orthogonal to the functions $j_{1},\dots,j_{n}$. For
$l=1,\dots,n-1$ we have 
\begin{multline*}
\int_{\mathbb{R}_{+}^{n}}\left[\frac{1}{3}\bar{R}_{ijkl}(q)z_{k}z_{l}+R_{ninj}(q)t^{2}\right]\partial_{ij}^{2}Uj_{b}\\
=\int_{\mathbb{R}_{+}^{n}}\left[\frac{1}{3}\bar{R}_{ijkl}(q)z_{k}z_{l}+R_{ninj}(q)t^{2}\right]\partial_{ij}^{2}U\partial_{l}Udzdt=0
\end{multline*}
by symmetry, since the integrand is odd with respect to the $z$ variables. 

For the last term, since when $i\neq j$ we have 
\[
\partial_{ij}U=\frac{n(n-2)z_{i}z_{j}}{\left((1+t)^{2}+|z|^{2}\right)^{\frac{n+2}{2}}}
\]
and since when $i=j$ we have $\bar{R}_{iikl}=0$ and, by (\ref{eq:Ricci}),
$R_{nini}=R_{nn}=0$, we have 
\begin{multline*}
\int_{\mathbb{R}_{+}^{n}}\left[\frac{1}{3}\bar{R}_{ijkl}(q)z_{k}z_{l}+R_{ninj}(q)t^{2}\right]\partial_{ij}^{2}UUdzdt\\
=\sum_{i\neq j}\sum_{k}\int_{\mathbb{R}_{+}^{n}}\left[\frac{1}{3}\bar{R}_{ijkl}(q)z_{k}z_{l}+R_{ninj}(q)t^{2}\right]\frac{n(n-2)z_{i}z_{j}}{\left((1+t)^{2}+|z|^{2}\right)^{n}}
\end{multline*}
and since $i\neq j$, by symmetry all the terms containing $t^{2}z_{i}z_{j}$
vanishes and the others terms are non zero only when $i=k$ and $j=l$
or when $j=k$ and $i=l$, thus 
\[
\int_{\mathbb{R}_{+}^{n}}\left[\frac{1}{3}\bar{R}_{ijkl}(q)z_{k}z_{l}+R_{ninj}(q)t^{2}\right]\partial_{ij}^{2}UUdzdt
\]
\[
=\sum_{k}\int_{\mathbb{R}_{+}^{n}}\left[\frac{1}{3}\bar{R}_{klkl}(q)+\frac{1}{3}\bar{R}_{lkkl}(q)\right]\frac{n(n-2)z_{k}^{2}z_{l}^{2}}{\left((1+t)^{2}+|z|^{2}\right)^{-n}}=0
\]
since $\bar{R}_{klkl}(q)=-\bar{R}_{lkkl}(q)$. Moreover
\begin{multline*}
\int_{\mathbb{R}_{+}^{n}}\left[\frac{1}{3}\bar{R}_{ijkl}(q)z_{k}z_{l}+R_{ninj}(q)t^{2}\right]\partial_{ij}^{2}Uy_{b}\partial_{b}Udtz\\
=n(2-n)\sum_{i\ne j}\sum_{k,s}\int_{\mathbb{R}_{+}^{n}}\left[\frac{1}{3}\bar{R}_{ijkl}(q)z_{k}z_{l}+R_{ninj}(q)t^{2}\right]\frac{z_{i}z_{j}\left(z_{s}z_{s}+t(1+t)\right)}{\left((1+t)^{2}+|z|^{2}\right)^{-n-1}}\\
=n(2-n)\sum_{k}\int_{\mathbb{R}_{+}^{n}}\left[\frac{1}{3}\bar{R}_{klkl}(q)+\frac{1}{3}\bar{R}_{lkkl}(q)\right]\frac{z_{k}^{2}z_{l}^{2}\left(\sum_{s}z_{s}z_{s}+t(1+t)\right)}{\left((1+t)^{2}+|z|^{2}\right)^{-n-1}}=0.
\end{multline*}
Then there exists a solution. Also there exists a unique solution
$v_{q}$ which is $L^{2}$-orthogonal to $j_{b}$ for $b=1,\cdots,n$. 

To prove the estimates (\ref{eq:Uvq}) and (\ref{new}) we use the
inversion $F:\mathbb{R}_{+}^{n}\rightarrow B^{n}\smallsetminus\left\{ (0,\dots,0-1)\right\} $,
where $B^{n}\subset\mathbb{R}^{n}$ is the closed ball centered in
$(0,\dots,0,-1/2)$ and radius $1/2$. The explicit expression for
$F$ is 
\[
F(y_{1},\dots,y_{n})=\frac{(y_{1},\dots,y_{n-1},y_{n}+1)}{y_{1}^{2}+\dots+y_{n-1}^{2}+(y_{n}+1)^{2}}+(0,\dots,0-1).
\]
We set
\[
f_{q}(F(y))=\left[\frac{1}{3}\bar{R}_{ijkl}(q)y_{k}y_{l}+R_{ninj}(q)y_{n}^{2}\right]\partial_{ij}^{2}U(y)U^{-\frac{n+2}{n-2}}(y).
\]
By direct computation we have $|f_{i}(F(y))|\le C(1+|y|)^{4}$, so
we have 
\begin{equation}
|f_{q}(\xi)|\le C\left(1+\frac{1}{|\xi|}\right)^{4}\le C\frac{1}{\left(1+|\xi|\right)^{4}}\label{eq:stimafi}
\end{equation}
So it is possible to smoothly extend $f_{q}$ to the whole $B^{n}$,
and it turns out that if $\gamma_{q}$ solves (\ref{eq:vqdef}), then
\textbf{$\bar{\gamma}_{q}:=(U^{-1}\gamma_{q})\circ F^{-1}$ }solves
\begin{equation}
\left\{ \begin{array}{ccc}
-\Delta\bar{\gamma}=f_{q} &  & \text{on }B^{n}\\
\frac{\partial\bar{\gamma}}{\partial y_{n}}+2\bar{\gamma}=0 &  & \text{on }\partial B^{n}
\end{array}\right..\label{eq:vbardef}
\end{equation}
Then existence and uniqueness of $\bar{\gamma}_{q}$ are standard.
To prove the decadence estimates, fixed $w\in B^{n}$, consider the
Green's function $G(\xi,w)$ with boundary condition $\left(\frac{\partial}{\partial\nu}+2\right)G=0$.
Then by Green's formula and by (\ref{eq:vbardef}) we have 
\[
\bar{\gamma}_{q}(\xi)=\int_{B^{n}}G(\xi,w)\Delta\bar{\gamma}_{q}(\xi)+\int_{\partial B^{n}}\bar{\gamma}_{q}\frac{\partial}{\partial\nu}G-G\frac{\partial}{\partial\nu}\bar{\gamma}_{q}=-\int_{B^{n}}G(\xi,w)f_{q}(\xi)
\]
 and, in light of (\ref{eq:stimafi}) we have 
\[
|\bar{\gamma}_{q}(\xi)|\le C\int_{B^{n}}|\xi-w|^{2-n}\left(1+|\xi|\right)^{-4}
\]
and by (\ref{eq:ALstimagreen}), since $n\ge5$ we get that $|\bar{\gamma}_{q}(\xi)|\le C\left(1+|\xi|\right)^{-2}$,
and by the definition of $\bar{\gamma}_{q}$ we deduce 
\[
|\gamma_{q}(y)|\le C\left(1+|y|\right)^{4-n}.
\]
The estimates on the first and the second derivatives of $\gamma_{q}$
can be achieved in a similar way. 

To prove (\ref{eq:Uvq}) and (\ref{new}) notice that, changing of
variables and proceeding as at the beginning of this proof, we have
\[
\int_{B^{n}}f_{q}(\xi)d\xi=\int_{\mathbb{R}_{+}^{n}}\frac{\left[\frac{1}{3}\bar{R}_{ijkl}(q)y_{k}y_{l}+R_{ninj}(q)y_{n}^{2}\right]\partial_{ij}^{2}U(y)U^{-\frac{n+2}{n-2}}(y)}{(y_{1}^{2}+\dots+y_{n-1}^{2}+(y_{n}+1)^{2})^{n}}dy=0.
\]
So we have, using (\ref{eq:vbardef}) and integrating by parts, that
\begin{equation}
0=\int_{B^{n}}f_{q}=-\int_{B^{n}}\Delta\bar{\gamma}_{q}=-\int_{\partial B^{n}}\frac{\partial}{\partial\nu}\bar{\gamma}_{q}=-\int_{\partial B^{n}}2\bar{\gamma}_{q}\label{eq:fq0}
\end{equation}
and, changing variables again, 
\begin{multline*}
0=\int_{\partial B^{n}}2\bar{\gamma}_{q}(\xi)d\xi_{1}\dots d\xi_{n-1}=\int_{\partial\mathbb{R}_{+}^{n}}U^{-1}(y)\gamma_{q}(y)U^{\frac{2(n-1)}{n-2}}(y)dy_{1}\dots dy_{n-1}\\
=\int_{\partial\mathbb{R}_{+}^{n}}U^{\frac{n}{n-2}}(y)\gamma_{q}(y)dy_{1}\dots dy_{n-1}.
\end{multline*}
It is known (see \cite{Al}), that on $H^{1}(B^{n})$ it holds 
\[
\inf_{\int_{\partial B_{n}}\phi=0}\frac{\int_{B^{n}}|\nabla\phi|^{2}}{\int_{\partial B^{n}}|\phi|^{2}}=2.
\]
Since, by (\ref{eq:fq0}), we know that $\int_{\partial B^{n}}\bar{\gamma}_{q}=0$,
we get
\[
2\int_{\partial B^{n}}\bar{\gamma}_{q}^{2}\le\int_{B^{n}}|\nabla\bar{\gamma}_{q}|^{2},
\]
so, integrating by parts
\[
-\int_{B^{n}}\bar{\gamma}_{q}\Delta\bar{\gamma}_{q}=\int_{B^{n}}|\nabla\bar{\gamma}_{q}|^{2}-2\int_{\partial B^{n}}\bar{\gamma}_{q}^{2}\ge0.
\]
By the properties of the inversion $F$ (see \cite[formula (5.10)]{Al})
we have also 
\[
-\int_{B^{n}}\bar{\gamma}_{q}\Delta\bar{\gamma}_{q}=-\int_{\mathbb{R}_{+}^{n}}\gamma_{q}\Delta\gamma_{q}.
\]
For claim (\ref{eq:dervq}) we refer to \cite[Proposition 5.1]{Al}.

To prove that $\gamma_{q}\in C^{2}(\partial M)$, we fix $q_{0}\in\partial M$.
If $q\in\partial M$ is sufficiently close to $q_{0}$, in Fermi coordinates
we have $q=q(\eta)=\exp_{q_{0}}\eta$, with $\eta\in\mathbb{R}^{n-1}$.
So $\gamma_{q}=\gamma_{\exp_{q_{0}}\eta}$ and we define 
\[
\Gamma_{i}=\left.\frac{\partial}{\partial y_{i}}\gamma_{\exp_{q_{0}}\eta}\right|_{\eta=0}.
\]
We prove the result for $\Gamma_{1}$, being the other cases completely
analogous. By (\ref{eq:vqdef}) we have that $\Gamma_{1}$ solves
\[
\left\{ \begin{array}{ccc}
-\Delta\Gamma_{1}=\left[\frac{1}{3}\left.\frac{\partial}{\partial\eta_{1}}\left(\bar{R}_{ijkl}(q(y))\right)\right|_{y=0}y_{k}y_{l}+\left.\frac{\partial}{\partial\eta_{1}}\left(\bar{R}_{ninj}(q(y))\right)\right|_{y=0}\right]\partial_{ij}^{2}U &  & \text{on }\mathbb{R}_{+}^{n};\\
\frac{\partial\Gamma_{1}}{\partial t}+nU^{\frac{2}{n-2}}\Gamma_{1}=0 &  & \text{on \ensuremath{\partial}}\mathbb{R}_{+}^{n}.
\end{array}\right.
\]
and, since $\frac{\partial R_{nn}}{\partial\eta_{i}}(q)=0$ (see \cite[Prop 3.2 (4)]{M1}),
we can proceed as at the beginning of this proof to show that $\Gamma_{1}$
exists. Analogously we get the claim for the second derivative. 
\end{proof}
\begin{rem}
\label{rem:Iam}We collect here some result contained in \cite[Lemma 9.4]{Al}
and in \cite[Lemma 9.5]{Al}. The proof is by direct computation.
For $m>k+1$
\begin{align}
\int_{0}^{\infty}\frac{t^{k}dt}{(1+t)^{m}} & =\frac{k!}{(m-1)(m-2)\cdots(m-1-k)}\label{eq:t-integrali}\\
\int_{0}^{\infty}\frac{dt}{(1+t)^{m}} & =\frac{1}{m-1}\nonumber 
\end{align}
Moreover, set, for $\alpha,m\in\mathbb{N}$, 
\[
I_{m}^{\alpha}:=\int_{0}^{\infty}\frac{s^{\alpha}ds}{\left(1+s^{2}\right)^{m}}
\]
it holds
\begin{align}
I_{m}^{\alpha}=\frac{2m}{\alpha+1}I_{m+1}^{\alpha+2} & \text{ for }\alpha+1<2m\label{eq:Iam}\\
I_{m}^{\alpha}=\frac{2m}{2m-\alpha-1}I_{m+1}^{\alpha} & \text{ for }\alpha+1<2m\nonumber \\
I_{m}^{\alpha}=\frac{2m-\alpha-3}{\alpha+1}I_{m}^{\alpha+2} & \text{ for }\alpha+3<2m.\nonumber 
\end{align}
\end{rem}
\begin{lem}
\label{lem:integraligamma}We have 
\begin{align*}
\int_{\mathbb{R}_{+}^{n}}L_{1}(y)|\bar{y}|^{4}y_{n}^{2}dy= & \omega_{n-2}\frac{n+1}{n}\frac{12}{(n-3)(n-4)(n-5)(n-6)}I_{n}^{n};\\
\int_{\mathbb{R}_{+}^{n}}L_{1}(y)|\bar{y}|^{2}y_{n}^{4}dy= & \omega_{n-2}\frac{144}{n(n-2)(n-3)(n-4)(n-5)(n-6)}I_{n}^{n};\\
\int_{\mathbb{R}_{+}^{n}}L_{2}(y)|\bar{y}|^{2}y_{n}^{2}dy= & \omega_{n-2}\frac{20}{(n-3)(n-4)(n-5)(n-6)}I_{n}^{n};\\
\int_{\mathbb{R}_{+}^{n}}L_{2}(y)y_{n}^{4}dy= & \omega_{n-2}\frac{240}{(n-1)(n-2)(n-3)(n-4)(n-5)(n-6)}I_{n}^{n};\\
\int_{\mathbb{R}_{+}^{n}}L_{3}(y)|\bar{y}|^{2}dy= & \omega_{n-2}\frac{16(n-1)}{(n-3)(n-5)(n-6)}I_{n}^{n};\\
\int_{\mathbb{R}_{+}^{n}}L_{3}(y)y_{n}^{2}dy= & \omega_{n-2}\frac{32}{(n-3)(n-4)(n-5)(n-6)}I_{n}^{n}.
\end{align*}
\end{lem}
\begin{proof}
The proof can be obtained  performing firstly a change in polar coordinates
in $\mathbb{R}^{n-1}$, then the change $s=r/(y_{n}+1)$ and using
Remark \ref{rem:Iam}. We recall that $\omega_{n-2}$ is the $n-1$
dimensional spherical element.
\end{proof}


\begin{thebibliography}{10}
\bibitem{Al}S. Almaraz, \emph{A compactness theorem for scalar-flat
metrics on manifolds with boundary}, Calc. Var. \textbf{41} (2011)
341-386.

\bibitem{A2}S. Almaraz, \emph{Blow-up phenomena for scalar-flat metrics
on manifolds with boundary}, J. Differential Equations \textbf{251}
(2011), no. 7, 1813-1840.

\bibitem{A3}S. Almaraz, \emph{An existence theorem of conformal scalar-flat
metrics on manifolds with boundary}, Pacific J. Math. \textbf{248}
(2010), 1-22.

\bibitem{ALM}A. Ambrosetti, Y.Y. Li, A. Malchiodi, \emph{On the Yamabe
problem and the scalar curvature problems under boundary conditions}.
Math. Ann. \textbf{322} (2002), 667-699. 

\bibitem{Au1}T. Aubin, \emph{Equations differentielles non lineaires
et probleme de Yamabe concernant la courbure scalaire}, J. Math. Pures
Appl. \textbf{55} (1976), 269-296.

\bibitem{Au}T. Aubin, Some Nonlinear Problems in Riemannian Geometry.
Springer Monographs in Mathematics. Springer, Berlin (1998).

\bibitem{B}S. Brendle, \emph{Convergence of the Yamabe flow in dimension
6 and higher}, Invent. Math. \textbf{170} (2007), 541- 576. 

\bibitem{ch}S. S. Chen, \emph{Conformal deformation to scalar flat
metrics with constant mean curvature on the boundary in higher dimensions},
arxiv preprint https://arxiv.org/abs/0912.1302 (2010).

\bibitem{DK}M. Disconzi, M. Khuri, \emph{Compactness and non-compactness
for the Yamabe problem on manifolds with boundary}, J. Reine Angew.
Math. \textbf{724} (2017), 145-201. 

\bibitem{FA}V. Felli, M. Ould Ahmedou, \emph{Compactness results
in conformal deformations of Riemannian metrics on manifolds with
boundaries}, Math. Z. \textbf{244} (2003), 175-210.

\bibitem{GMP}M.G. Ghimenti, A.M. Micheletti, A. Pistoia, \emph{Blow-up
phenomena for linearly perturbed Yamabe problem on manifolds with
umbilic boundary, }J. Differential Equations, in press, arXiv:1804.05559.

\bibitem{GMP2}M. Ghimenti, A.M. Micheletti, A. Pistoia, \emph{Linear
perturbation of the Yamabe problem on manifolds with boundary}, J.
Geom. Anal. \textbf{28} (2018), 1315-1340. 

\bibitem{Gi}G. Giraud, \emph{Sur la problème de Dirichlet généralisé}.
Ann. Sci. Ècole Norm. Sup. \textbf{46}, (1929) 131-145.

\bibitem{HL}Z.C. Han, Y. Li, \emph{The Yamabe problem on manifolds
with boundary: existence and compactness results}. Duke Math, J. \textbf{99}
(1999), 489-542.

\bibitem{HV} E. Hebey, M. Vaugon, \emph{Le probleme de Yamabe equivariant},
Bull. Sci. Math. \textbf{117} (1993), 241-286 

\bibitem{KMS} M. Khuri, F. Marques, R. Schoen, \emph{A compactness
theorem for the Yamabe problem}. J. Differ. Geom. \textbf{81} (2009),
143-196 

\bibitem{Es}J. Escobar, \emph{Conformal deformation of a Riemannian
metric to a scalar flat metric with constant mean curvature on the
boundary}, Ann. Math. \textbf{136}, (1992), 1-50.

\bibitem{Es2}J. Escobar, \emph{Sharp constant in a Sobolev trace
inequality}, Indiana Univ. Math. J. \textbf{37}, (1988), 687-698.

\bibitem{LZ} Y.Y. Li, M. Zhu, \emph{Yamabe type equations on three
dimensional Riemannian manifolds}. Commun. Contemp. Math. \textbf{1}
(1999), 1-50.

\bibitem{MN}M. Mayer, C.B. Ndiaye, \emph{Barycenter technique and
the Riemann mapping problem of Cherrier-Escobar}. J. Differential
Geom. \textbf{107} (2017), no. 3, 519-560. 

\bibitem{M1}F. Marques, \emph{Existence results for the Yamabe problem
on manifolds with boundary}, Indiana Univ. Math. J. \textbf{54} (2005)
1599-1620. 

\bibitem{M3}F. Marques, \emph{A priori estimates for the Yamabe problem
in the non-locally conformally flat case}, J. Differ. Geom. \textbf{71}
(2005) 315-346. 

\bibitem{M2}F. Marques, \emph{Compactness and non compactness for
Yamabe-type problems}, Progress in Nonlinear Differential Equation
and Their Applications, \textbf{86} (2017) 121-131.

\bibitem{SZ}R. Schoen, D. Zhang, \emph{Prescribed scalar curvature
on the n-sphere}, Calc. Var. Partial Differ. Equ. \textbf{4} (1996),
1-25.

\bibitem{S}R. Schoen, \emph{Conformal deformation of a Riemannian
metric to constant scalar curvature}, J. Differ. Geom. \textbf{20}
(1984), 479-495.

\bibitem{T}N. Trudinger, \emph{Remarks concerning the conformal deformation
of Riemannian structures on compact manifolds}, Annali Scuola Norm.
Sup. Pisa \textbf{22} (1968), 265-274.

\bibitem{Y}H. Yamabe, \emph{On a deformation of Riemannian structures
on compact manifolds}, Osaka Math. J. \textbf{12} (1960), 21-37.
\end{thebibliography}
\end{document}